\documentclass[a4paper,twoside,10pt]{article}

\usepackage[english]{babel}
\usepackage[latin1]{inputenc}
\usepackage{lmodern}
\usepackage[T1]{fontenc}
\usepackage{indentfirst}
\usepackage[a4paper,top=3cm,bottom=3cm,left=3cm,right=3cm,%
bindingoffset=5mm]{geometry}
\usepackage{lmodern}
\usepackage[T1]{fontenc}
\usepackage{lettrine}
\usepackage{booktabs}
\usepackage{epsfig}
\usepackage{epstopdf}
\usepackage{emp}
\usepackage{amsthm}
\usepackage{amsmath,amssymb,amsthm}
\usepackage{braket}
\usepackage{chngpage}
\usepackage{cases}
\usepackage{graphicx}
\usepackage{graphics}
\usepackage{tabularx}
\usepackage{multirow}
\usepackage{color}
\theoremstyle{remark}
\newtheorem{remark}{Remark}[section]

\theoremstyle{remark}
\newtheorem{test}{Test}[section]

\theoremstyle{plain}
\newtheorem{theorem}{Theorem}[section]
\newtheorem{proposition}{Proposition}[section]

\newtheorem{lemma}{Lemma}[section]

\newcommand{\numberset}{\mathbb}
\newcommand{\N}{\numberset{N}}
\newcommand{\B}{\numberset{B}}
\newcommand{\Pk}{\numberset{P}}
\newcommand{\R}{\numberset{R}}

\newcommand{\K}{\numberset{K}}

\renewcommand{\epsilon}{\varepsilon}
\renewcommand{\theta}{\vartheta}
\renewcommand{\rho}{\varrho}
\renewcommand{\phi}{\varphi}
{\left\lbrace\begin{array}{@{}l@{}}}%
{\end{array}\right.}

\theoremstyle{remark}

\author{Giuseppe Vacca  
  \thanks{Dipartimento di Matematica e Applicazioni,  Universit\`a degli Studi di Milano Bicocca, Via Roberto Cozzi 55 - 20125 Milano, Italy; E-mail: giuseppe.vacca@unimib.it.} 
  }

\title{An $H^1$-conforming  Virtual Element Methods for  Darcy equations and  Brinkman equations}

\date{}

\begin{document}

\maketitle
\begin{abstract} The focus of the present paper is on developing a Virtual Element Method for Darcy and Brinkman equations.
In \cite{stokes}  we presented a family of Virtual Elements for Stokes equations and we defined a new Virtual Element space of velocities  such that the associated discrete kernel is pointwise divergence-free.  We use a slightly different Virtual Element space having two fundamental properties: the $L^2$-projection onto $\mathbb{P}_k$ is exactly computable on the basis of the degrees of freedom, and the
associated discrete kernel is still pointwise divergence-free. 
The resulting numerical scheme for the Darcy equation has optimal order of convergence and $H^1$ conforming velocity solution. 
We can apply the same approach to develop a robust virtual element method for the Brinkman equation that is stable for both the Stokes and Darcy limit case.
We provide a rigorous error analysis of the method and several numerical tests.
 \end{abstract}
%
%\begin{resume} ... \end{resume}
%
%\subjclass{65N30, 65N12, 65N15, 76D07}
%
%\keywords{Virtual element method, Polygonal meshes, Darcy equations, Brinkman equations}
%

\maketitle

% -------------------------------------------------------------------
\section{Introduction}
\label{sec:1}
% -------------------------------------------------------------------

The \textbf{Virtual Element Methods} (in short, VEM or VEMs) is a recent technique for solving PDEs. VEMs were recently introduced in \cite{VEM-volley} as a generalization of the finite element method on polyhedral or polygonal meshes. 
In the numerical analysis and engineering literature there has been a recent growth of interest 
in developing numerical methods that 
can make use of general polygonal and polyhedral meshes, as opposed to more standard triangular/quadrilateral (tetrahedral/hexahedral) 
grids. Indeed, making use of polygonal meshes brings forth a range of advantages, including for instance automatic {hanging node treatment}, 
more efficient approximation of geometric data features, better domain meshing 
capabilities, more efficient and easier 
adaptivity, more robustness to mesh deformation, and others. This interest in the literature is also reflected in commercial codes, 
such as CD-Adapco, that have recently included polytopal meshes.

We refer to the recent papers and monographs 
\cite{%
BLS05bis,BLM11book,Bishop13,JA12,LMSXX,RW12,POLY37,%
ST04,TPPM10,VW13,Wachspress11,DiPietro-Ern-1,di2016discontinuous,Gillette-1,PolyDG-1} % 
as a brief representative sample of the increasing list of technologies that make use of polygonal/polyhedral meshes. We mention 
here in particular the polygonal finite elements, that generalize finite elements to polygons/polyhedrons by making use of 
generalized non-polynomial shape functions, and the mimetic discretisation schemes \cite{lopezvacca, da2015symplectic}, that combine ideas from the finite difference and 
finite element methods.

The principal idea behind VEM is to use approximated discrete bilinear forms that require only integration of polynomials on the (polytopal) 
element in order to be computed. The resulting discrete solution is conforming and the accuracy granted by such discrete bilinear forms 
turns out to be sufficient to achieve the correct order of convergence. 
Following this approach, VEM is able to make use of very general polygonal/polyhedral meshes without the need to integrate 
complex non-polynomial functions on the elements and without loss of accuracy. Moreover, 
VEM is not restricted to low order converge and can be easily applied to three dimensions and use non convex (even non simply connected) elements.
The Virtual Element Method has been developed successfully for a large range of problems, see for instance \cite{VEM-volley,Brezzi:Marini:plates,VEM-elasticity,VEM-enhanced,BM13,BFMXX,VEM-stream,GTP14,BBPSXX,benedetto2016hybrid,mora2015virtual, hpvem, lovadina,caceres2015mixed,wriggers,plates-zhao,vaccabis, vaccahyper, giani, gardini, frittelli}. 
A helpful paper for the implementation of the method is \cite{VEM-hitchhikers}.

The focus of this paper is on developing a new Virtual Element Method for the Darcy equation that is suitable for a robust extension to the (more complex) Brinkman problem. 
For such a problem, other VEM numerical schemes  have been proposed, see for example \cite{BFMXX, da2016mixed}. 
%%
%In \cite{VEM-elasticity} the authors presented a family of Virtual Elements for the linear elasticity problem that are locking-free in the incompressible limit. 
%As a consequence, the scheme in \cite{VEM-elasticity} can be immediately extended to the Stokes problem, thus yielding a stable VEM family that would be comparable to the Crouzeix-Raviart finite element family. 

In \cite{stokes} the authors developed a new Virtual Element Method for  Stokes problems by exploiting  the flexibility of the Virtual Element construction in a new way. In particular, they define a new Virtual Element space of velocities carefully designed to solve the Stokes problem. In connection with a suitable pressure space, the new Virtual Element space leads to an exactly divergence-free discrete velocity, a favorable property when more complex problems, such as the Navier-Stokes problem, are considered. We highlight that this feature is not shared by the method defined in \cite{VEM-elasticity} or by most of the standard mixed Finite Element methods, where the divergence-free constraint is imposed only in a weak (relaxed) sense. 
%We however remark that, using different discretization methodologies, some divergence-free methods have already proposed in the literature (for instance, see  \cite{cockburndivfree1, cockburndivfree2, cockburndivfree3, divfree1, divfree2,Buffa2011}).

In the present contribution we develop  the Virtual Element Method for Darcy equations by introducing a slightly different virtual space for the velocities such that the local $L^2$ orthogonal  projection onto the space of polynomials of degree less or equal than $k$ (where $k$ is the polynomial degree of accuracy of the method) can be computed using the local degrees of freedom. The resulting Virtual Elements family inherits the advantages on the scheme proposed in \cite{stokes}, in particular it yields an exactly divergence-free discrete kernel. Thus we obtain a stable Darcy element that is also uniformly stable for the Stokes problem. A sample of uniformly stable methods for Darcy-Stokes model is for instance  \cite{mardal2002robust, xie2008uniformly, karper2009unified, vassilev}.

The last part of the paper deals with the analysis of a new mixed finite element method for  Brinkman equations that stems from the above scheme for the Darcy problem. Mathematically, the Brinkman problem resembles both the Stokes problem for fluid flow and the Darcy problem for flow in porous media (see \cite{stenberg, anaya2015priori, neilan}).
Constructing finite element methods to solve the Brinkman equation that are robust for both (Stokes and Darcy) limits is challenging. We will see how the above Virtual Element approach offers a natural and straightforward
framework for constructing stable numerical algorithms for the Brinkman
equations.

We remark that the proposed scheme belongs to the class of the \textit{pressure-robust} method, i.e. delivers a velocity error independent of the continuous pressure.

%%%%%%%%%%%%%%%%%%%%%%%%%%%%%%%%%%%%%%%%%%%%%%%%%%%%%%%%%%%%%%

%%%%%%%%%%%%%%%%%%%%%%%%%%%%%%%%%%%%%%%%%%%%%%%%%%%%%%%%%%%%%%

The paper is organized as follows. 
In Section \ref{sec:2} we introduce the model continuous Darcy problem. 
In Section \ref{sec:3} we present its VEM discretisation. 
In Section \ref{sec:4} we detail the theoretical features and the convergence analysis of the problem. 
In Section \ref{sec:6} we  develop a stable numerical methods for  Brinkman equations.
In Section \ref{sec:7} we show the numerical tests.
Finally in the Appendix we present the theoretical analysis of the extension to the Darcy equation of the scheme of  \cite{VEM-elasticity}. Even though this latter method is not recommended for the Darcy problem, the numerical  experiments showed an unexpected optimal convergence rate for the pressure. We theoretically prove this behaviour, developing an inverse inequality for the VEM space, which is interesting on its own. 

% -----------------------------------------------------------------------------------------------------------------------------
% -----------------------------------------------------------------------------------------------------------------------------

\section{The continuous problem}
\label{sec:2}

We consider the classical Darcy equation that describes the flow of a fluid through a porous medium.
Let $\Omega \subseteq \R^2$ be a bounded polygon then the Darcy equation in \textbf{mixed form} is
\begin{equation}
\label{eq:darcy primale}
\left\{
\begin{aligned}
& \text{find $(\mathbf{u}, p)$ such that} \\
&  \K^{-1}  \mathbf{u}  + \nabla p = \mathbf{0} \qquad  & &\text{in $\Omega$,} \\
& {\rm div} \, \mathbf{u} = f \qquad & &\text{in $\Omega$,} \\
& \mathbf{u} \cdot \mathbf{n} = 0  \qquad & &\text{on $ \partial \Omega$,}
\end{aligned}
\right.
\end{equation}
where $\mathbf{u}$ and $p$ are respectively the velocity and the pressure fields, $f \in L^2(\Omega)$ is the source term and $\K$ is a uniformly symmetric, positive definite tensor that represents the permeability of the medium. 
%With standard notation we denote with
%\begin{equation}
%\label{eq:tensornorm}
%\| \K \| \qquad \text{and} \qquad \| \K^{-1} \|
%\end{equation}
%the matrix norm of $\K$ and $\K ^{-1}$, respectively. 
%
From \eqref{eq:darcy primale}, since we have assumed no flux boundary
conditions all over $\partial \Omega$, the external force $f$ has zero mean value on $\Omega$. 
We consider the spaces
\begin{equation*}
%\label{eq:spazi continui}
\mathbf{V}:= \left\{ \mathbf{u} \in H({\rm div}, \Omega), \quad \text{s.t} \quad \mathbf{u} \cdot \mathbf{n} = 0 \quad \text{on $\partial \Omega$} \right\}, \qquad Q:= L^2_0(\Omega) = \left\{ q \in L^2(\Omega) \quad \text{s.t.} \quad \int_{\Omega} q \,{\rm d}\Omega = 0 \right\} 
\end{equation*}
equipped with the natural norms
\begin{equation*}
%\label{eq:norme continue}
\| \mathbf{v}\|_{\mathbf{V}}^2 := \| \mathbf{v}\|_{\left[L^2(\Omega) \right]^2}^2   +  \| {\rm div} \, \mathbf{v}\|_{L^2(\Omega)}^2 \quad , \qquad 
\|q\|_Q := \| q\|_{L^2(\Omega)},
\end{equation*}
and the bilinear forms $a(\cdot, \cdot) \colon \mathbf{V} \times \mathbf{V} \to \R$ and $b(\cdot, \cdot) \colon \mathbf{V} \times Q \to \R$ defined by:
\begin{equation}
\label{eq:forma a}
a (\mathbf{u},  \mathbf{v}) := \int_{\Omega} \K^{-1} \,  \mathbf{u} \cdot   \mathbf{v} \,{\rm d} \Omega, \qquad \text{for all $\mathbf{u},  \mathbf{v} \in \mathbf{V}$}
\end{equation}
\begin{equation}
\label{eq:forma b}
b(\mathbf{v}, q) :=  \int_{\Omega} {\rm div} \,\mathbf{v}\, q \,{\rm d}\Omega \qquad \text{for all $\mathbf{v} \in \mathbf{V}$, $q \in Q$.}
\end{equation}
Then the \textbf{variational formulation} of Problem \eqref{eq:darcy primale} is
\begin{equation}
\label{eq:darcy variazionale}
\left\{
\begin{aligned}
& \text{find $(\mathbf{u}, p) \in \mathbf{V} \times Q$, such that} \\
& a(\mathbf{u}, \mathbf{v}) + b(\mathbf{v}, p) = \mathbf{0} \qquad & \text{for all $\mathbf{v} \in \mathbf{V}$,} \\
&  b(\mathbf{u}, q) = (f, q) \qquad & \text{for all $q \in Q$,}
\end{aligned}
\right.
\end{equation}
where
\[
 (f, q) := \int_{\Omega} f\, q \, {\rm d} \Omega \qquad \text{for all $q \in Q$.}
\]
Let us introduce the kernel
\begin{equation*}
%\label{eq:Z}
\mathbf{Z} := \{\mathbf{v} \in \mathbf{V} \quad \text{s.t.}  \quad  b(\mathbf{v}, q) =0 \quad \text{for all $q \in Q$}\};
\end{equation*}
then it is straightforward to see that
\[
\| \mathbf{v}\|_{\mathbf{V}} := \| \mathbf{v}\|_{\left[L^2(\Omega) \right]^2} \qquad \text{for all $\mathbf{v} \in \mathbf{Z}$.}
\]
It is well known that (see for instance \cite{fortin1991mixed}):
\begin{itemize}
\item $a(\cdot, \cdot)$ and $b(\cdot, \cdot)$ are continuous, i.e.
\[
|a(\mathbf{u}, \mathbf{v})| \leq \|a\| \|\mathbf{u}\|_{\mathbf{V}}\|\mathbf{v}\|_{\mathbf{V}} \qquad \text{for all $\mathbf{u}, \mathbf{v} \in \mathbf{V}$,}
\] 
\[
|b(\mathbf{v}, q)| \leq \|b\| \|\mathbf{v}\|_{\mathbf{V}} \|q\|_Q \qquad \text{for all $\mathbf{v} \in \mathbf{V}$ and $q \in Q$;} 
\]
\item $a(\cdot, \cdot)$ is coercive on the kernel $\mathbf{Z}$, i.e. there exists a positive constant $\alpha$ depending on $\K$ such that 
\begin{equation}
\label{eq:coercive}
a(\mathbf{v}, \mathbf{v}) \geq \alpha \|\mathbf{v}\|^2_{\mathbf{V}} \qquad \text{for all $\mathbf{v} \in \mathbf{Z}$;}
\end{equation}
\item $b(\cdot,\cdot) $ satisfies the inf-sup condition, i.e.   
\begin{equation}
\label{eq:inf-sup}
\exists \, \beta >0 \quad \text{such that} \quad \sup_{\mathbf{v} \in \mathbf{V} \, \mathbf{v} \neq \mathbf{0}} \frac{b(\mathbf{u}, q)}{ \|\mathbf{v}\|_{\mathbf{V}}} \geq \beta \|q\|_Q \qquad \text{for all $q \in Q$.}
\end{equation}
\end{itemize}
Therefore, Problem \eqref{eq:darcy variazionale} has a unique solution $(\mathbf{u}, p) \in \mathbf{V} \times Q$ such that
\[
\|\mathbf{u}\|_{\mathbf{V}} + \|p\|_Q \leq C \, \|f\|_{L^2(\Omega)}
\]
with the constant $C$ depending only on $\Omega$ and $\K$.

% -----------------------------------------------------------------------------------------------------------------------------
% -----------------------------------------------------------------------------------------------------------------------------

\section{Virtual formulation for Darcy equations}
\label{sec:3}
% -----------------------------------------------------------------------------------------------------------------------------
% -----------------------------------------------------------------------------------------------------------------------------

\subsection{Decomposition and the original virtual element spaces}
\label{sub:3.1}
 
We outline the Virtual Element discretization of Problem \eqref{eq:darcy variazionale}. 
Here and in the rest of the paper the symbol $C$ will indicate a generic positive constant independent of the mesh size that may change at 
each occurrence. Moreover, given any subset $\omega$ in ${\mathbb R}^2$ and $k \in {\mathbb N}$, we will denote by $\Pk_k(\omega)$ the 
polynomials of total degree at most $k$ defined on $\omega$, with the extended notation $\Pk_{-1}(\omega)=\emptyset$.
Let $\set{\mathcal{T}_h}_h$ be a sequence of decompositions of $\Omega$ into general polygonal elements $K$ with
\[
 h_K := {\rm diameter}(K) , \quad
h := \sup_{K \in \mathcal{T}_h} h_K .
\]
We suppose that for all $h$, each element $K$ in $\mathcal{T}_h$ fulfils the following assumptions:
\begin{itemize}
\item $\mathbf{(A1)}$ $K$ is star-shaped with respect to a ball of radius $ \ge\, \gamma \, h_K$, 
\item $\mathbf{(A2)}$ the distance between any two vertexes of $K$ is $\ge c \, h_K$, 
\end{itemize}
where $\gamma$ and $c$ are positive constants. We remark that the hypotheses above, though not too restrictive in many practical cases, 
can be further relaxed, as noted in ~\cite{VEM-volley, 2016stability}. 
From now on we assume that $\K$ is  piecewise constant with respect to $\mathcal{T}_h$ on $\Omega$. 

Using standard VEM notation, for $k \in \N$, let us define the spaces 
\begin{itemize}
\item $\Pk_k(K)$ the set of polynomials on $K$ of degree $\le k$,
\item $\B_k(K) := \{v \in C^0(\partial K) \quad \text{s.t} \quad v_{|e} \in \Pk_k(e) \quad \forall\mbox{ edge } e \subset \partial K\}$,
\item $\mathcal{G}_{k}(K):= \nabla(\Pk_{k+1}(K)) \subseteq [\Pk_{k}(K)]^2$,
\item $\mathcal{G}_{k}(K)^{\perp} := \mathbf{x}^{\perp}[\Pk_{k-1}(K)] \subseteq [\Pk_{k}(K)]^2$ with $\mathbf{x}^{\perp}:= (x_2, -x_1)$.
\end{itemize}
% When the set $\omega$ is clear from the contest, we use $\Pk_k$, $\mathcal{G}_{k-1}$ and $\mathcal{G}_{k-1}^{\perp}$. 
In \cite{stokes} the authors have introduced a new family of Virtual Elements for the Stokes problem
on polygonal meshes. In particular, by a proper choice of the Virtual space of velocities, the virtual local spaces are associated to a Stokes-like variational problem on each element. The main ideas of the method are
\begin{itemize}
\item the Virtual space contains the space of all the polynomials of the prescribed order plus suitable non polynomial functions,
\item the degrees of freedom are carefully chosen so that the $H^1$ semi-norm projection onto the space of polynomials can be exactly computed,
\item the choice of the Virtual space of velocities and the associated degrees of freedom guarantee that the final discrete velocity is pointwise divergence-free and more generally the discrete kernel is contained in the continuous one.
\end{itemize}
In this section we briefly  recall from \cite{stokes} the notations, the main properties of the Virtual spaces and some details of the construction of the  $H^1$ semi-norm projection.
Let $k \geq 2$ the polynomial degree of accuracy of the method, then we define  on each element $K \in \mathcal{T}_h$ the finite dimensional local virtual space
\begin{multline}
\label{eq:W_h^K}
\mathbf{W}_h^K := \biggl\{  
\mathbf{v} \in [H^1(K)]^2 \quad \text{s.t} \quad \mathbf{v}_{|{\partial K}} \in [\B_k(\partial K)]^2 \, , \biggr.
\\
\left.
\biggl\{
\begin{aligned}
& - \boldsymbol{\Delta}    \mathbf{v}  -  \nabla s \in \mathcal{G}_{k-2}(K)^{\perp},  \\
& {\rm div} \, \mathbf{v} \in \Pk_{k-1}(K),
\end{aligned}
\biggr. \qquad \text{ for some $s \in L^2(K)$}
\quad \right\}
\end{multline}
where all the operators and equations above are to be interpreted in the distributional sense. 
It is easy to check that
$
[\Pk_k(K)]^2 \subseteq \mathbf{W}_h^K$,  
and that (see \cite{stokes} for the proof) the dimension of $\mathbf{W}_h^K$ is
\begin{equation}
\label{eq:dimensione W_h^K}
\begin{split}
\dim\left( \mathbf{W}_h^K \right) &= \dim\left([\B_k(\partial K)]^2\right) + \dim\left(\mathcal{G}_{k-2}(K)^{\perp}\right) + \left( \dim(\Pk_{k-1}(K)) - 1\right) \\
&= 2n_K k + \frac{(k-1)(k-2)}{2}  + \frac{(k+1)k}{2} - 1.
\end{split}
\end{equation}
The corresponding degrees of freedom are chosen prescribing, given a function $\mathbf{v} \in \mathbf{W}_h^K$, the following linear operators $\mathbf{D_V}$, split into four subsets (see Figure \ref{fig:dofsloc}):
\begin{itemize}
\item $\mathbf{D_V1}$:  the values of $\mathbf{v}$ at the vertices of the polygon $K$,
\item $\mathbf{D_V2}$: the values of $\mathbf{v}$ at $k-1$ distinct points of every edge $e \in \partial K$ (for example we can take the $k-1$ internal points of the $(k+1)$-Gauss-Lobatto quadrature rule  in $e$, as suggested in \cite{VEM-hitchhikers}),
\item $\mathbf{D_V3}$: the moments of $\mathbf{v}$
\[
\int_K \mathbf{v} \cdot \mathbf{g}_{k-2}^{\perp} \, {\rm d}K \qquad \text{for all $\mathbf{g}_{k-2}^{\perp} \in \mathcal{G}_{k-2}(K)^{\perp}$,}
\]
\item $\mathbf{D_V4}$: the moments up to order $k-1$ and greater than zero of ${\rm div} \,\mathbf{v}$ in $K$, i.e.
\[
\int_K ({\rm div} \,\mathbf{v}) \, q_{k-1} \, {\rm d}K \qquad \text{for all $q_{k-1} \in \Pk_{k-1}(K) / \R$.}
\] 
\end{itemize} 

\begin{figure}[!h]
\center{
\includegraphics[scale=0.20]{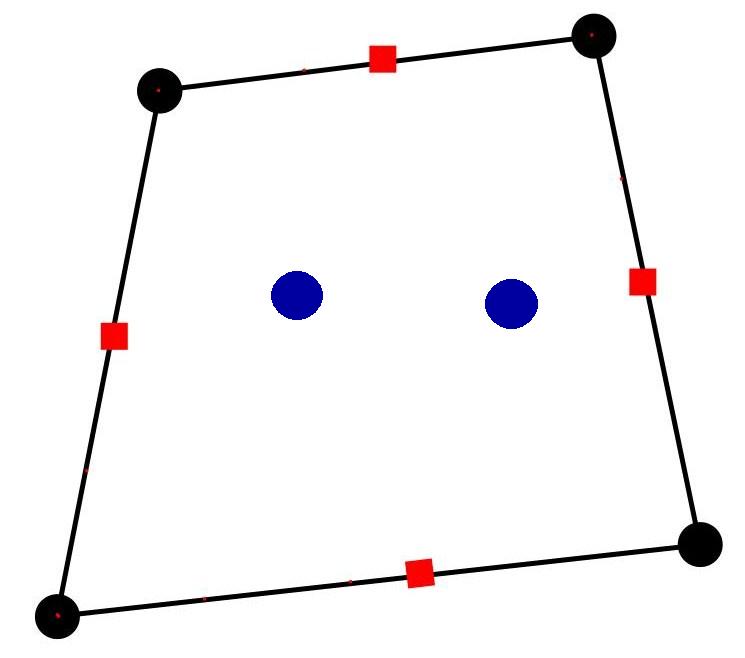} \qquad \qquad
\includegraphics[scale=0.20]{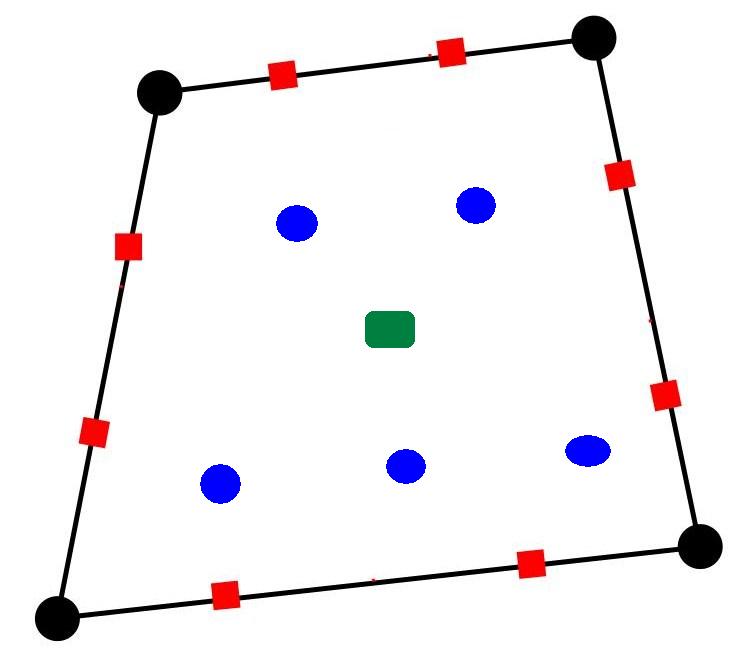}
\caption{Degrees of freedom for $k=2$, $k=3$. We denote $\mathbf{D_V1}$ with the black dots, $\mathbf{D_V2}$ with the red squares, $\mathbf{D_V3}$ with the green rectangles, $\mathbf{D_V4}$ with the blue dots inside the element.}
\label{fig:dofsloc}
}
\end{figure}
For all $K \in \mathcal{T}_h$, we introduce the $H^1$ semi-norm projection ${\Pi}_{k}^{\nabla,K} \colon \mathbf{W}_h^K \to [\Pk_k(K)]^2$, defined by 
\begin{equation}
\label{eq:Pi_k^K}
\left\{
\begin{aligned}
& \int_K \boldsymbol{\nabla} \,\mathbf{q}_k : \boldsymbol{\nabla} (\mathbf{v}_h - \, {\Pi}_{k}^{\nabla,K}   \mathbf{v}_h) \, {\rm d} K = 0 \qquad  \text{for all $\mathbf{q}_k \in [\Pk_k(K)]^2$,} \\
& \Pi_0^{0,K}(\mathbf{v}_h - \,  {\Pi}_{k}^{\nabla,K}  \mathbf{v}_h) = \mathbf{0} \, ,
\end{aligned}
\right.
\end{equation} 
where $\Pi_0^{0,K}$ is the $L^2$-projection operator onto the constant functions defined on $K$. It is immediate to check that the energy projection is well defined and it clearly holds ${\Pi}_{k}^{\nabla,K} \mathbf{q}_k = \mathbf{q}_k$ for all $\mathbf{q}_k \in \Pk_k(K)$.
Moreover  the operator ${\Pi}_{k}^{\nabla,K}$ is computable in terms of the degrees of freedom $\mathbf{D_V}$ (see equations $(27)-(29)$ in \cite{stokes} and the subsequent discussion).

% -----------------------------------------------------------------------------------------------------------------------------
% -----------------------------------------------------------------------------------------------------------------------------

\subsection{The modified virtual space and the projection $\Pi^{0,K}_{k}$}
\label{sub:3.2}
Let $n$ a positive integer, then for all $K \in \mathcal{T}_h$, the $L^2$-projection $\Pi^{0,K}_{n} \colon \mathbf{W}^K_h \to [\Pk_{n}(K)]^2$ is defined by
\[
\int_K  \mathbf{q}_{n} \cdot (\mathbf{v}_h - \Pi^{0,K}_{n} \mathbf{v}_h)  \, {\rm d} K = 0 \qquad \text{for all $\mathbf{q}_n \in  [\Pk_{n}(K)]^2$.}
\]
It is possible to check (see Section 3.3 of \cite{stokes} for the proof) that the degrees of freedom $\mathbf{D_V}$ allow us to compute exactly the $L^2$-projection $\Pi^{0,K}_{k-2}$.  On the other hand we can observe that we can not  compute exactly  from the DoFs  the $L^2$-projection onto the space of polynomials of degree $\leq k$.
The goal of the present section is to introduce, taking the inspiration from \cite{VEM-enhanced}, a new virtual space $\mathbf{V}_h^K$ to be used in place of $\mathbf{W}_h^K$ in such a way that
\begin{itemize}
\item the DoFs $\mathbf{D_V}$ can still be used for $\mathbf{V}_h^K$,
\item $[\Pk_k(K)]^2 \subseteq \mathbf{V}_h^K$,
\item the projection $\Pi_k^{0,K}\colon \mathbf{V}_h^K \to [\Pk_k(K)]^2$ can be exactly computable by the DoFs $\mathbf{D_V}$.
\end{itemize}

To construct $\mathbf{V}_h^K$ we proceed as follows: first of all we define an \textbf{augmented virtual local space} $\mathbf{U}_h^K$ by taking
\begin{multline*}
%\label{eq:enalrged}
\mathbf{U}_h^K := \biggl\{  
\mathbf{v} \in [H^1(K)]^2 \quad \text{s.t} \quad \mathbf{v}_{|{\partial K}} \in [\B_k(\partial K)]^2 \, , \biggr.
\\
\left.
\biggl\{
\begin{aligned}
& - \boldsymbol{\Delta}    \mathbf{v}  -  \nabla s \in \mathcal{G}_{k}(K)^{\perp},  \\
& {\rm div} \, \mathbf{v} \in \Pk_{k-1}(K),
\end{aligned}
\biggr. \qquad \text{ for some $s \in L^2(K)$}
\quad \right\}
\end{multline*}
Now we define the \textbf{enhanced Virtual Element space} $\mathbf{V}_h^K$ as the restriction  of $\mathbf{U}_h^K$ given by
\begin{equation}
\label{eq:V_h^K}
\mathbf{V}_h^K := \left\{ \mathbf{v} \in \mathbf{U}_h^K \quad \text{s.t.} \quad   \left(\mathbf{v} - \Pi^{\nabla,K}_k \mathbf{v}, \, \mathbf{g}_k^{\perp} \right)_{[L^2(K)]^2} = 0 \quad \text{for all $\mathbf{g}_k^{\perp} \in  \mathcal{G}_{k}(K)^{\perp}/\mathcal{G}_{k-2}(K)^{\perp}$} \right\} ,
\end{equation}
where the symbol $\mathcal{G}_{k}(K)^{\perp}/\mathcal{G}_{k-2}(K)^{\perp}$ denotes the polynomials in $\mathcal{G}_{k}(K)^{\perp}$ that are $L^2-$orthogonal to all polynomials of $\mathcal{G}_{k-2}(K)^{\perp}$.
We proceed by investigating the dimension and by choosing suitable DoFs of the virtual space $\mathbf{V}_h$. First of all we recall from  \cite{conforming} the following facts 
\begin{equation}
\label{eq:dimensioni}
 \dim\left([\B_k(\partial K)]^2\right) = 2n_K k, \quad  \dim\left(\Pk_{k-1}(K)\right) = \frac{k(k+1)}{2}, \quad  \dim\left(\mathcal{G}_{k}(K)^{\perp}\right) = \frac{k(k+1)}{2} 
\end{equation}
where $n_K$ is the number of edges of the polygon $K$.

\begin{lemma}
\label{lm:augmented}
The dimension of $\mathbf{U}_h^K$ is
\begin{equation*}
%\label{eq:dimensione tildeV_h^K}
\begin{split}
\dim\left( \mathbf{U}_h^K \right) 
&= 2n_K k + \frac{k(k+1)}{2}  + \frac{(k+1)k}{2} - 1.
\end{split}
\end{equation*}
Moreover as DoFs for $\mathbf{U}_h^K$ we can take the linear operators $\mathbf{D_V}$ and plus the moments 
\[
\mathbf{D_U} \colon \int_K \mathbf{v} \cdot \mathbf{g}_{k}^{\perp} \, {\rm d}K \qquad \text{for all $\mathbf{g}_{k}^{\perp} \in \mathcal{G}_{k}(K)^{\perp}/\mathcal{G}_{k-2}(K)^{\perp}$.}
\] 
\end{lemma}

\begin{proof}
The proof is virtually identical to that given in \cite{stokes} for $\mathbf{W}_h^K$ and it is based (see for instance \cite{fortin1991mixed}) on the fact that given
\begin{itemize}
\item a polynomial function $\mathbf{g}_b \in [\B_k({\partial K})]^2$,
\item a polynomial function $\mathbf{h} \in \mathcal{G}_{k}(K)^{\perp}$,
\item a polynomial function $g \in \Pk_{k-1}(K)$  satisfying the compatibility condition
\begin{equation*}
%\label{compcond}
\int_{K} g \, {\rm d}\Omega = \int_{\partial K} \mathbf{g}_b \cdot \mathbf{n} \, {\rm d}s, 
\end{equation*}
\end{itemize}
there exists a unique pair $(\mathbf{v}, s)\in  \mathbf{U}_h^K\times  L^2(K) / \R$ such that
\begin{equation}\label{datasol}
\mathbf{v}_{|\partial K} = \mathbf{g}_b, \quad {\rm div} \, \mathbf{v} = g, \quad 
 -  \boldsymbol{\Delta}    \mathbf{v}  -  \nabla s = \mathbf{h}.
\end{equation}
Moreover, since from \cite{conforming},  ${\rm rot} \colon \mathcal{G}_{k}(K)^{\perp} \to \Pk_{k-1}(K)$
%, where  ${\rm rot}$ is the rotational operator in 2D, i.e. the rotated divergence,  
is an isomorphism, 
we can conclude that  the map that associates a given compatible data set   
$(\mathbf{g}_{b}, \, \mathbf{h}, \, g)$ to the velocity field $\mathbf{v}$ that solves~\eqref{datasol} is an injective map. Then
\[
\dim\left( \mathbf{U}_h^K \right) = \dim\left([\B_k(\partial K)]^2\right) + \dim\left(\mathcal{G}_{k}(K)^{\perp}\right) + \left( \dim(\Pk_{k-1}(K)) - 1\right)
\]
and the thesis follows from \eqref{eq:dimensioni}.
\end{proof}

\begin{proposition}
\label{prop:enhanced}
The dimension of $\mathbf{V}_h^K$ is equal to that of $\mathbf{W}_h^K$ that is, as in \eqref{eq:dimensione W_h^K}
\begin{equation}
\label{eq:dimensione V_h^K}
\begin{split}
\dim\left( \mathbf{V}_h^K \right) &=  2n_K k + \frac{(k-1)(k-2)}{2}  + \frac{(k+1)k}{2} - 1.
\end{split}
\end{equation}
As DoFs in $\mathbf{V}_h^K$ we can take $\mathbf{D_V}$.
\end{proposition}

\begin{proof}
From \eqref{eq:dimensioni} it is straightforward to check that
\[
\dim\left(\mathcal{G}_{k}(K)^{\perp}/ \mathcal{G}_{k-2}(K)^{\perp}\right) = \dim\left(\mathcal{G}_{k}(K)^{\perp}\right) - \dim\left(\mathcal{G}_{k-2}(K)^{\perp}\right) = 2k -1.
\]
Hence, neglecting the independence of the additional $2k - 1$ conditions in \eqref{eq:V_h^K}, it holds that
\begin{equation}
\label{eq:dimension}
\dim\left( \mathbf{V}_h^K \right) \geq \dim\left( \mathbf{U}_h^K \right) - (2k -1) =  2n_K k + \frac{(k-1)(k-2)}{2}  + \frac{(k+1)k}{2} - 1 = \dim\left( \mathbf{W}_h^K \right).
\end{equation}
We now observe that a function $\mathbf{v} \in \mathbf{V}_h^K$ such that $\mathbf{D_V}(\mathbf{v}) = 0$ is identically zero. Indeed,  from \eqref{eq:Pi_k^K}, it is immediate to check that in this case the $\Pi_k^{\nabla, K} \, \mathbf{v}$ would be zero, implying that all its
moment are zero, in particular, since $\mathbf{v} \in \mathbf{V}_h^K$, all the moments $\mathbf{D_U}$ of $\mathbf{v}$ are also zero. Now, from Lemma \ref{lm:augmented}, we have that $\mathbf{v}$ is zero.
Therefore, from \eqref{eq:dimension}, we obtain that the dimension of $\mathbf{V}_h^K$ is actually the same of  $\mathbf{W}_h^K$, and that the DoFs $\mathbf{D_V}$ are unisolvent for $\mathbf{V}_h^K$.
\end{proof}

\begin{proposition}
\label{prp:pkprojection}
The degrees of freedom $\mathbf{D_V}$ allow us to compute exactly the $L^2$-projection $\Pi^{0,K}_{k} \colon \mathbf{V}_h \to [\Pk_{k}(K)]^2$, i.e. the moments
\[
\int_K \mathbf{v} \cdot \mathbf{q}_{k} \, {\rm d} K
\]
for all $\mathbf{v} \in \mathbf{V}_h$ and for all $\mathbf{q}_{k} \in [\Pk_{k}(K)]^2$. 
\end{proposition}
\begin{proof}
Let us set
\begin{equation*}
%\label{eq:dcomp}
\mathbf{q}_{k} = \nabla q_{k+1} + \mathbf{g}_{k-2}^{\perp} + \mathbf{g}_{k}^{\perp}.
\end{equation*}
with $q_{k+1} \in \Pk_{k+1}(K) / \R$, $\mathbf{g}_{k-2}^{\perp} \in \mathcal{G}_{k-2}^{\perp}(K)$ and $\mathbf{g}_{k}^{\perp} \in \mathcal{G}_{k}^{\perp}(K)/\mathcal{G}_{k-2}^{\perp}(K)$. Therefore using the Green formula and since $\mathbf{v} \in \mathbf{V}_h$, we get 
\[
\begin{split}
\int_K \mathbf{v} \cdot \mathbf{q}_{k} \, {\rm d} K & = \int_K \mathbf{v} \cdot (\nabla q_{k+1} +  \mathbf{g}_{k-2}^{\perp} + \mathbf{g}_{k}^{\perp}) \, {\rm d} K \\
& = - \int_K {\rm div} \, \mathbf{v} \, q_{k+1} \, {\rm d} K  +
\int_K \mathbf{v} \cdot  \mathbf{g}_{k-2}^{\perp} \, {\rm d} K +
\int_K \Pi_k^{\nabla, K}\, \mathbf{v} \cdot  \mathbf{g}_{k}^{\perp} \, {\rm d} K  
+ \int_{\partial K} q_{k+1} \, \mathbf{v} \cdot \mathbf{n} \, {\rm d} s.
\end{split}
\]
Now, since ${\rm div} \, \mathbf{v}$ is a polynomial of degree less or equal than $k-1$ we can reconstruct its value from $\mathbf{D_V4}$ and compute exactly the first term. The second term is computable from $\mathbf{D_V3}$. The third term is computable from all the $\mathbf{D_V}$ using the projection $\Pi_k^{\nabla, K}\, \mathbf{v}$. Finally from $\mathbf{D_V1}$ and $\mathbf{D_V2}$ we can reconstruct $\mathbf{v}$ on the boundary and so compute exactly the boundary term. 
\end{proof}
For what concerns the pressures we take the standard finite dimensional space
\begin{equation}
\label{eq:Q_h^K}
Q_h^K := \Pk_{k-1}(K) 
\end{equation}
having dimension
\begin{equation*}
%\label{eq:dimensione Q_h^K}
\dim(Q_h^K) = \dim(\Pk_{k-1}(K))  = \frac{(k+1)k}{2}.
\end{equation*}
The corresponding degrees of freedom are chosen defining for each $q\in Q_h^K$ the following linear operators $\mathbf{D_Q}$:
\begin{itemize}
\item $\mathbf{D_Q}$: the moments up to order $k-1$ of $q$, i.e.
\[
\int_K q \, p_{k-1} \, {\rm d}K \qquad \text{for all $p_{k-1} \in \Pk_{k-1}(K)$.}
\]
\end{itemize}
Finally we define the global virtual element spaces as
\begin{equation}
\label{eq:V_h}
\mathbf{V}_h := \{ \mathbf{v} \in [H^1(\Omega)]^2  \quad \text{s.t} \quad \mathbf{v} \cdot \mathbf{n} = 0 \quad \text{on $\partial \Omega$} \quad \text{and} \quad \mathbf{v}_{|K} \in \mathbf{V}_h^K  \quad \text{for all $K \in \mathcal{T}_h$} \}
\end{equation} 
and
\begin{equation}
\label{eq:Q_h}
Q_h := \{ q \in L_0^2(\Omega) \quad \text{s.t.} \quad q_{|K} \in  Q_h^K \quad \text{for all $K \in \mathcal{T}_h$}\},
\end{equation}
with the obvious associated sets of global degrees of freedom. A simple computation shows that:
\begin{equation*}
%\label{eq:Vdofs}
\dim(\mathbf{V}_h) = n_P \left( \frac{(k+1)k}{2} -1  +  \frac{(k-1)(k-2)}{2} \right) 
+ 2(n_V + (k-1) n_E) + (n_{V, B} + (k-1) n_{E, B})
\end{equation*}
and
\begin{equation*}
%\label{eq:Qdofs}
\dim(Q_h) = n_P \frac{(k+1)k}{2} - 1 ,
\end{equation*}
where $n_P$ is the number of elements, $n_E$, $n_V$ (resp., $n_{E, B}$, $n_{V, B}$) is the number of  internal edges and vertexes (resp., boundary edges and vertexes) in $\mathcal{T}_h$.
As observed in \cite{stokes}, we remark that
\begin{equation}\label{eq:divfree}
{\rm div}\, \mathbf{V}_h\subseteq Q_h .
\end{equation}
\begin{remark}
By definition \eqref{eq:V_h} it is clear that our discrete velocities field is $H^1$-conforming, in particular we obtain continuous velocities, whereas the natural discretization is only $H({\rm div})$- conforming. This property, in combination with \eqref{eq:divfree}, will make our method suitable for a (robust) extension to the Brinkman problem.
\end{remark}
%

%\begin{remark}
%As observed in \cite{stokes}, the space $\mathcal{G}_{k-2}(K)^{\perp}$ that defines the degrees of freedom $\boldsymbol{D_V3}$ can be replaced by 
%any space $ \mathcal{G}_{k-2}(K)^{\oplus} \subseteq [\Pk_{k-2}(K)]^2$ that satisfies 
%\[
%[\Pk_{k-2}(K)]^2 = \mathcal{G}_{k-2}(K) \oplus \mathcal{G}_{k-2}(K)^{\oplus}.
%\]
%An example is given by 
%the space $\mathcal{G}_{k-2}(K)^{\oplus}:= \mathbf{x}^{\perp} [\Pk_{k-3}(K)]$ with $\mathbf{x}^{\perp}:= (x_2, -x_1)$.
%\end{remark}

% -----------------------------------------------------------------------------------------------------------------------------
% -----------------------------------------------------------------------------------------------------------------------------
\subsection{The discrete bilinear forms}
\label{sub:3.3}

The next step in the construction of our method is to define on the virtual spaces $\mathbf{V}_h$ and $Q_h$ a discrete version of the bilinear forms $a(\cdot, \cdot)$ and $b(\cdot, \cdot)$ given in \eqref{eq:forma a} and \eqref{eq:forma b}.
For simplicity we assume that the tensor $\K$ is piecewise constant with respect to the decomposition $\mathcal{T}_h$, i.e. $\K$ is constant on each polygon $K \in \mathcal{T}_h$.
First of all we decompose into local contributions the bilinear forms $a(\cdot,\cdot)$ and $b(\cdot, \cdot)$, the norms $\|\cdot\|_{\mathbf{V}}$ and $\|\cdot \|_Q$ by defining 
\begin{equation*}
%\label{eq:forma a locali continue}
a (\mathbf{u},  \mathbf{v}) =: \sum_{K \in \mathcal{T}_h} a^K (\mathbf{u},  \mathbf{v}) \qquad \text{for all $\mathbf{u},  \mathbf{v} \in \mathbf{V}$}
\end{equation*}
\begin{equation*}
%\label{eq:forma b locali continue}
b (\mathbf{v},  q) =: \sum_{K \in \mathcal{T}_h} b^K (\mathbf{v},  q) \qquad \text{for all $\mathbf{v} \in \mathbf{V}$ and $q \in Q$,}
\end{equation*}
and 
\begin{equation*}
%\label{eq:norme locali}
\|\mathbf{v}\|_{\mathbf{V}} =: \left(\sum_{K \in \mathcal{T}_h} \|\mathbf{v}\|^2_{\mathbf{V}, K}\right)^{1/2} \quad \text{for all $\mathbf{v} \in \mathbf{V}$,} \qquad \|q\|_Q =: \left(\sum_{K \in \mathcal{T}_h} \|q\|^2_{Q, K}\right)^{1/2} \quad \text{for all $q \in Q$.}
\end{equation*}
We now define discrete versions of the bilinear form $a(\cdot, \cdot)$ (cf.~\eqref{eq:forma a}), and of the bilinear form $b(\cdot, \cdot)$ (cf.~\eqref{eq:forma b}). For what concerns $b(\cdot, \cdot)$, we simply  set 

\begin{equation}\label{bhform}
b(\mathbf{v}, q) = \sum_{K \in \mathcal{T}_h} b^K(\mathbf{v}, q) = \sum_{K \in \mathcal{T}_h}  \int_K {\rm div} \, \mathbf{v} \, q \,{\rm d}K \qquad \text{for all $\mathbf{v} \in \mathbf{V}_h$, $q \in Q_h$},
\end{equation}
i.e. as noticed in \cite{stokes} we do not introduce any approximation of the bilinear form. We notice that~\eqref{bhform}  
is computable from the degrees of freedom $\mathbf{D_V1}$, $\mathbf{D_V2}$ and $\mathbf{D_V4}$, since $q$ is polynomial in each element $K \in \mathcal{T}_h$. 
On the other hand, the bilinear form $a(\cdot, \cdot)$ needs to be dealt with in a more careful way.
First of all, by Proposition \ref{prp:pkprojection}, we observe that  for all $\mathbf{q}_k \in [\Pk_k(K)]^2 $ and for all $\mathbf{v}\in \mathbf{V}_h^K $, the quantity  
\begin{equation*}
%\label{eq:aKcomp}
a^K (\mathbf{q}_k,  \mathbf{v}) = \int_{K} \K^{-1} \,  \mathbf{q}_k \cdot \mathbf{v} \,{\rm d}K.
\end{equation*}
is exactly computable by the DoFs.
However, for an arbitrary pair $(\mathbf{u},\mathbf{v} )\in \mathbf{V}_h^K \times \mathbf{V}_h^K $, the quantity $a_h^K(\mathbf{w}, \mathbf{v})$ is clearly not computable.  
In the standard procedure of VEM framework, we define a computable discrete local bilinear form
\begin{equation}
\label{eq:a_h^K} 
a_h^K(\cdot, \cdot) \colon \mathbf{V}_h^K \times \mathbf{V}_h^K \to \R
\end{equation}
approximating the continuous form $a^K(\cdot, \cdot)$ and satisfying the following properties:
\begin{itemize}
\item $\mathbf{k}$\textbf{-consistency}: for all $\mathbf{q}_k \in [\Pk_k(K)]^2$ and $\mathbf{v}_h \in \mathbf{V}_h^K$
\begin{equation}\label{eq:consist}
a_h^K(\mathbf{q}_k, \mathbf{v}_h) = a^K( \mathbf{q}_k, \mathbf{v}_h);
\end{equation}
\item \textbf{stability}:  there exist  two positive constants $\alpha_*$ and $\alpha^*$, independent of $h$ and $K$, such that, for all $\mathbf{v}_h \in \mathbf{V}_h^K$, it holds
\begin{equation}\label{eq:stabk}
\alpha_* a^K(\mathbf{v}_h, \mathbf{v}_h) \leq a_h^K(\mathbf{v}_h, \mathbf{v}_h) \leq \alpha^* a^K(\mathbf{v}_h, \mathbf{v}_h).
\end{equation}
\end{itemize}
Let $\mathcal{R}^K \colon \mathbf{V}_h^K \times \mathbf{V}_h^K \to \R$ be a (symmetric) stabilizing bilinear form, satisfying
\begin{equation}
\label{eq:R^K}
c_* a^K(\mathbf{v}_h, \mathbf{v}_h) \leq  \mathcal{R}^K(\mathbf{v}_h, \mathbf{v}_h) \leq c^* a^K(\mathbf{v}_h, \mathbf{v}_h) \qquad \text{for all $\mathbf{v}_h \in \mathbf{V}_h$ such that ${\Pi}_{k}^{0,K} \mathbf{v}_h= \mathbf{0}$}
\end{equation}
with $c_*$ and $c^*$  positive  constants independent of $h$ and $K$.
Then, we can set
\begin{equation}
\label{eq:a_h^K def}
a_h^K(\mathbf{u}_h, \mathbf{v}_h) := a^K \left({\Pi}_{k}^{0,K} \mathbf{u}_h, {\Pi}_{k}^{0,K} \mathbf{v}_h \right) + \mathcal{R}^K \left((I -{\Pi}_{k}^{0,K}) \mathbf{u}_h, (I -{\Pi}_{k}^{0,K}) \mathbf{v}_h \right)
\end{equation}
for all $\mathbf{u}_h, \mathbf{v}_h \in \mathbf{V}_h^K$.

It is straightforward to check that Definition~\eqref{eq:Pi_k^K} and properties~\eqref{eq:R^K} imply the consistency and the stability of the bilinear form $a_h^K(\cdot, \cdot)$.

\begin{remark}
In the construction of the stabilizing form $\mathcal{R}^K$ with condition \eqref{eq:R^K}  we essentially require that the stabilizing term $\mathcal{R}^K(\mathbf{v}_h, \mathbf{v}_h)$ scales as $a^K(\mathbf{v}_h, \mathbf{v}_h)$. 
Following the standard VEM technique (cf. \cite{VEM-volley, VEM-hitchhikers} for more details), denoting with $\bar{\mathbf{u}}_h$, $\bar{\mathbf{v}}_h \in \R^{N_K}$
the vectors containing the values of the $N_K$ local degrees of freedom associated to $\mathbf{u}_h, \mathbf{v}_h \in \mathbf{V}_h^K$, we set
\[
\mathcal{R}^K (\mathbf{u}_h, \mathbf{v}_h) =  \alpha^K \,  \bar{\mathbf{u}}_h^T \bar{\mathbf{v}}_h ,
\] 
where $\alpha^K$ is a suitable positive constant that scales as $|K|$. For example, in the numerical tests presented in Section \ref{sec:7}, we have chosen $\alpha^K$ as the mean value of the eigenvalues of the matrix stemming from the  
term $a^K \left({\Pi}_{k}^{0,K} \mathbf{u}_h,\,  {\Pi}_{k}^{0,K} \mathbf{v}_h \right) $ in  \eqref{eq:a_h^K def}.
\end{remark}

Finally we define the global approximated bilinear form $a_h(\cdot, \cdot) \colon \mathbf{V}_h \times \mathbf{V}_h \to \R$ by simply summing the local contributions:
\begin{equation}
\label{eq:a_h}
a_h(\mathbf{u}_h, \mathbf{v}_h) := \sum_{K \in \mathcal{T}_h}  a_h^K(\mathbf{u}_h, \mathbf{v}_h) \qquad \text{for all $\mathbf{u}_h, \mathbf{v}_h \in \mathbf{V}_h$.}
\end{equation}
% -----------------------------------------------------------------------------------------------------------------------------
% -----------------------------------------------------------------------------------------------------------------------------

% -----------------------------------------------------------------------------------------------------------------------------

\subsection{The discrete problem}\label{sec:discrete}
\label{sub:3.4}
We are now ready to state the proposed discrete problem. Referring to~\eqref{eq:V_h}, \eqref{eq:Q_h},  ~\eqref{bhform}, and  ~\eqref{eq:a_h}  we consider the \textbf{virtual element problem}:
\begin{equation}
\label{eq:darcy virtual}
\left\{
\begin{aligned}
& \text{find $(\mathbf{u}_h, p_h) \in \mathbf{V}_h \times Q_h$, such that} \\
& a_h(\mathbf{u}_h, \mathbf{v}_h) + b(\mathbf{v}_h, p_h) = 0 \qquad & \text{for all $\mathbf{v}_h \in \mathbf{V}_h$,} \\
&  b(\mathbf{u}_h, q_h) = (f, q_h) \qquad & \text{for all $q_h \in Q_h$.}
\end{aligned}
\right.
\end{equation}

We point out that the symmetry of $a_h(\cdot, \cdot)$ together with \eqref{eq:stabk}  easily implies that $a_h(\cdot, \cdot)$  is  (uniformly) continuous with respect to the  $L^2$ norm. 
Moreover, as observed in \cite{stokes}, introducing the discrete kernel:
\begin{equation*}
%\label{eq:Z_h}
\mathbf{Z}_h := \{ \mathbf{v}_h \in \mathbf{V}_h \quad \text{s.t.} \quad b(\mathbf{v}_h, q_h) = 0 \quad \text{for all $q_h \in Q_h$}\},
\end{equation*}
it is immediate to check that
\begin{equation*}
%\label{kernincl}
\mathbf{Z}_h \subseteq \mathbf{Z} .
\end{equation*}
Then the bilinear form $a_h(\cdot, \cdot)$ is also uniformly coercive on the discrete kernel $\mathbf{Z}_h$ with respect to the  $\mathbf{V}$ norm.
%
%Therefore, the existence and the uniqueness of the solution to 
%Problem~\eqref{eq:darcy virtual} follows will follow if a suitable inf-sup condition is fulfilled.
%
Moreover as a direct consequence of Proposition 4.3 in \cite{stokes}, we have the following stability result.
\begin{proposition}
\label{thm2} Given the discrete spaces
$\mathbf{V}_h$ and $Q_h$ defined in~\eqref{eq:V_h} and~\eqref{eq:Q_h}, there exists a positive $\tilde{\beta}$, independent of $h$, such that:
\begin{equation}
\label{eq:inf-sup discreta}
\sup_{\mathbf{v}_h \in \mathbf{V}_h \, \mathbf{v}_h \neq \mathbf{0}} \frac{b(\mathbf{v}_h, q_h)}{ \|\mathbf{v}_h\|_{\mathbf{V}}} \geq \tilde{\beta} \|q_h\|_Q \qquad \text{for all $q_h \in Q_h$.}
\end{equation}
\end{proposition}
In particular, the  the inf-sup condition of Proposition~\ref{thm2}, along with property~\eqref{eq:divfree}, implies that:
\begin{equation*}
%\label{eq:divfree2}
{\rm div}\, \mathbf{V}_h = Q_h .
\end{equation*}
Finally we can state the well-posedness of virtual problem \eqref{eq:darcy virtual}.
\begin{theorem}
Problem \eqref{eq:darcy virtual} has a unique solution $(\mathbf{u}_h, p_h) \in \mathbf{V}_h \times Q_h$, verifying the estimate
\[
\|\mathbf{u}_h\|_{\mathbf{V}} + \|p_h\|_{Q} \leq C \|f\|_{0}.
\]
\end{theorem}

% -----------------------------------------------------------------------------------------------------------------------------
% -----------------------------------------------------------------------------------------------------------------------------

\section{Theoretical results}
\label{sec:4}

We begin by proving an approximation result for the virtual local space $\mathbf{V}_h$. First of all, let us recall a classical result by Brenner-Scott (see \cite{MR2373954}).

\begin{lemma}
\label{lm:scott}
Let $K \in \mathcal{T}_h$, then for all $\mathbf{u} \in [H^{s+1}(K)]^2$ with $0 \leq s \leq k$, there exists a  polynomial function $\mathbf{u}_{\pi} \in [{\Pk}_k(K)]^2$,  such that
\begin{equation}
\label{eq:scott}
\|\mathbf{u} - \mathbf{u}_{\pi}\|_{0,K} + h_K |\mathbf{u} -\mathbf{u}_{\pi} |_{1,K} \leq C h_K^{s+1}| \mathbf{u}|_{s+1,K}.
\end{equation}
\end{lemma}
We have the following approximation results (for the  proof see  \cite{navier-stokes}).
\begin{proposition}
\label{thm:interpolation}
Let $\mathbf{u} \in \mathbf{V} \cap [H^{s+1}(\Omega)]^2$ with $0 \leq s \leq k$. Under the assumption $\mathbf{(A1)}$ and $\mathbf{(A2)}$ 
on the decomposition $\mathcal{T}_h$, there exists $\mathbf{u}_{int} \in \mathbf{W}_h$  such that
\begin{equation*}
%\label{eq:interpolation}
\|\mathbf{u} - \mathbf{u}_{int}\|_{0} +  h_K |\mathbf{u} -\mathbf{u}_{int} |_{1,K} \leq C h_K^{s+1}| \mathbf{u}|_{s+1,K}.
\end{equation*}
where $C$ is a constant independent of $h$.
\end{proposition}
For what concerns the pressures, from classic polynomial approximation theory \cite{MR2373954}, for  $q \in H^k(\Omega)$ it holds
\begin{equation}
\label{eq: stime q_I}
\inf_{ \mathbf{q}_h\in \mathbf{Q}_h }\|q - q_h \|_{Q} \leq C \, h^k \, |q|_k.
\end{equation}
We are ready to state the following convergence theorem.
\begin{theorem}
\label{thm5}
Let $(\mathbf{u}, p) \in \mathbf{V} \times Q$ be the solution of problem \eqref{eq:darcy variazionale} and $(\mathbf{u}_h, p_h) \in \mathbf{V}_h \times Q_h$ be the solution of problem \eqref{eq:darcy virtual}. Then it holds
\[
\begin{gathered}
\|\mathbf{u} - \mathbf{u}_h\|_{0} \leq C \, h^{k+1} \, |\mathbf{u}|_{k+1} , \qquad \text{and} \qquad \|\mathbf{u} - \mathbf{u}_h\|_{\mathbf{V}} \leq C \, h^{k} \, |\mathbf{u}|_{k+1} , \\
\|p - p_h\|_{Q} \leq C \, h^{k} ( |\mathbf{u}|_{k+1} + |p|_{k}).
\end{gathered}
\]  
\end{theorem}

\begin{proof}
We begin by remarking that as a consequence of the inf-sup condition  with classical arguments (see for
instance Proposition 2.5 in \cite{fortin1991mixed}), there exists $\mathbf{u}_I \in  \mathbf{V}_h$ such that 
\begin{gather}
\label{eq:divinterpolant}
\Pi_{k-1}^{0,K} ({\rm div} \, \mathbf{u}_I) = {\rm div} \, \mathbf{u}_I = \Pi_{k-1}^{0,K}({\rm div} \, \mathbf{u}) \qquad \text{for all $K \in \mathcal{T}_h$,} \\
\label{eq:stability}
\|\mathbf{u} - \mathbf{u}_I\|_{0} \leq C \, \inf_{\mathbf{v}_h \in \mathbf{V}_h}\|\mathbf{u} - \mathbf{v}\|_{0} \qquad \text{and} \qquad
\qquad
\|\mathbf{u} - \mathbf{u}_I\|_{\mathbf{V}} \leq C \, \inf_{\mathbf{v}_h \in \mathbf{V}_h}\|\mathbf{u} - \mathbf{v}\|_{\mathbf{V}}.
\end{gather}
Let us set $\boldsymbol{\delta}_h = \mathbf{u}_I - \mathbf{u}_h$. From \eqref{eq:divinterpolant} and   \eqref{eq:darcy virtual}, we have that ${\rm div}\, \boldsymbol{\delta}_h = 0$ and thus  $\boldsymbol{\delta}_h \in \mathbf{Z}_h$. Now, using \eqref{eq:coercive}, \eqref{eq:stabk}, \eqref{eq:darcy virtual} and introducing the piecewise polynomial approximation \eqref{eq:scott} together with \eqref{eq:consist}, we have
\[
\begin{split}
\alpha_* \, \alpha \, \|\boldsymbol{\delta}_h\|^2_0 & \leq \alpha_* \, a(\boldsymbol{\delta}_h, \, \boldsymbol{\delta}_h) \leq a_h(\boldsymbol{\delta}_h, \, \boldsymbol{\delta}_h)  = a_h(\mathbf{u}_I, \, \boldsymbol{\delta}_h) - a_h(\mathbf{u}_h, \, \boldsymbol{\delta}_h) \\
& = a_h(\mathbf{u}_I, \, \boldsymbol{\delta}_h) + b(\boldsymbol{\delta}_h, p_h) =  a_h(\mathbf{u}_I, \, \boldsymbol{\delta}_h) \\
& = \sum_{K \in \mathcal{T}_h} a_h^K(\mathbf{u}_I, \, \boldsymbol{\delta}_h) =
   \sum_{K \in \mathcal{T}_h} \left ( a_h^K(\mathbf{u}_I - \mathbf{u}_{\pi}, \, \boldsymbol{\delta}_h) + a^K(\mathbf{u}_{\pi}\,, \boldsymbol{\delta}_h) \right) \\
& =  \sum_{K \in \mathcal{T}_h} \left ( a_h^K(\mathbf{u}_I - \mathbf{u}_{\pi}, \,  \boldsymbol{\delta}_h) + a^K(\mathbf{u}_{\pi} - \mathbf{u},  \, \boldsymbol{\delta}_h) \right) - a( \mathbf{u}, \, \boldsymbol{\delta}_h)  \\
& = \sum_{K \in \mathcal{T}_h} \left ( a_h^K(\mathbf{u}_I - \mathbf{u}_{\pi}, \,  \boldsymbol{\delta}_h) + a^K(\mathbf{u}_{\pi} - \mathbf{u},  \, \boldsymbol{\delta}_h) \right) + b( \boldsymbol{\delta}_h \, , p)  \\
& = \sum_{K \in \mathcal{T}_h} \left ( a_h^K(\mathbf{u}_I - \mathbf{u}_{\pi}, \,  \boldsymbol{\delta}_h) + a^K(\mathbf{u}_{\pi} - \mathbf{u},  \, \boldsymbol{\delta}_h) \right) \\
& \leq C \sum_{K \in \mathcal{T}_h} \left ( \|\mathbf{u}_I - \mathbf{u}_{\pi}\|_{0,K} + \|\mathbf{u} - \mathbf{u}_{\pi}\|_{0,K} \right)  \|\mathbf{\delta}_h\|_{0, K} \\
& \leq C \, \left ( \|\mathbf{u}_I - \mathbf{u}_{\pi}\|_{0} + \|\mathbf{u} - \mathbf{u}_{\pi}\|_{0} \right)  \|\mathbf{\delta}_h\|_{0}
\end{split}
\]
then
\[
\|\mathbf{\delta}_h\|_{0} \leq C \, \|\mathbf{u}_I - \mathbf{u}_{\pi}\|_{0} + \|\mathbf{u} - \mathbf{u}_{\pi}\|_{0} .
\] 
The $L^2$-estimate follows easily by the triangle inequality.
It is also straightforward to see from \eqref{eq:darcy variazionale} and \eqref{eq:darcy virtual} that
\[
b(\mathbf{u} - \mathbf{u}_h, q_h) = 0 \qquad \text{for all $q_h \in Q_h$},
\]
than we get ${\rm div}\, \mathbf{u}_h = \Pi_{k-1}^{0,K} ({\rm div}\, \mathbf{u})$ for all $K \in \mathcal{T}_h$ and therefore
\[
\|{\rm div} (\mathbf{u} - \mathbf{u}_h)\|_{0} = \sum_{K \in \mathcal{T}_h} \|{\rm div} \, \mathbf{u} - \Pi_{k-1}^{0,K} ({\rm div}\, \mathbf{u})\|_{0,K} \leq C \, h^k \, |{\rm div}\, \mathbf{u}|_{k} \leq C\, h^k |\mathbf{u}|_{k+1},
\]
from which the estimate in the ${\mathbf{V}}$ norm.
We proceed by analysing the error on the pressure field. 
Let $q_h \in Q_h$, then from the discrete inf-sup condition \eqref{eq:inf-sup discreta}, we infer:
\begin{equation}
\label{eq:thm4.1}
\tilde{\beta} \|p_h - q_h\|_Q \leq \sup_{\mathbf{v}_h \in \mathbf{V}_h \, \mathbf{v}_h \neq \mathbf{0}} \frac{b(\mathbf{v}_h, p_h - q_h)}{\|\mathbf{v}_h\|_{V}} = \sup_{\mathbf{v}_h \in \mathbf{V}_h \, \mathbf{v}_h \neq \mathbf{0}} \frac{b(\mathbf{v}_h, p_h - p) + b(\mathbf{v}_h,  p - q_h)}{\|\mathbf{v}_h\|_{V}}. 
\end{equation}
Since $(\mathbf{u},p)$ and $(\mathbf{u}_h,p_h)$ are respectively the solution of \eqref{eq:darcy variazionale} and \eqref{eq:darcy virtual}, it follows that
\begin{gather*}
a(\mathbf{u}, \mathbf{v}_h) + b(\mathbf{v}_h, p) = 0 \qquad  \text{for all $\mathbf{v}_h \in \mathbf{V}_h$,} \\
a_h(\mathbf{u}_h, \mathbf{v}_h) + b(\mathbf{v}_h, p_h) = 0 \qquad  \text{for all $\mathbf{v}_h \in \mathbf{V}_h$.} 
\end{gather*}
Therefore, we get
\begin{equation*}
%\label{eq:thm4.2}
b(\mathbf{v}_h, p_h - p)  = a(\mathbf{u}, \mathbf{v}_h) - a_h(\mathbf{u}_h, \mathbf{v}_h)  \qquad  \text{for all $\mathbf{v}_h \in \mathbf{V}_h$.}
\end{equation*}
Using \eqref{eq:consist}, the continuity of $a_h(\cdot, \cdot)$ and the triangle inequality, we get:
%the projection $\mathcal{P}_h$, see~\eqref{eq:ph}, to get:
\[
\begin{split}
b(\mathbf{v}_h, p_h - p) &= a(\mathbf{u}, \mathbf{v}_h)  - a_h(\mathbf{u}_h, \mathbf{v}_h)  
= \sum_{K \in \mathcal{T}_h} \biggl( a^K(\mathbf{u}, \mathbf{v}_h)  - a_h^K(\mathbf{u}_h, \mathbf{v}_h)\biggr) \\
&= \sum_{K \in \mathcal{T}_h} \biggl( a^K(\mathbf{u} - \mathbf{u}_{\pi}, \mathbf{v}_h)  + 
a_h^K(\mathbf{u}_{\pi}  -\mathbf{u}_h, \mathbf{v}_h) \biggr) \\
& \leq \sum_{K \in \mathcal{T}_h} C \bigl( \|\mathbf{u} - \mathbf{u}_{\pi}\|_{\mathbf{V},K} + 
\|(\mathbf{u}_{\pi} - \mathbf{u}_h)\|_{\mathbf{V},K}    \bigr) \|\mathbf{v}_h\|_{\mathbf{V},K} \\
&  \leq  \sum_{K \in \mathcal{T}_h} C  \bigl( \|\mathbf{u} - \mathbf{u}_{\pi}\|_{\mathbf{V},K} + \|\mathbf{u} -\mathbf{u}_h\|_{\mathbf{V},K} \bigr) 
\|\mathbf{v}_h\|_{\mathbf{V},K} 
\end{split}
\]
where $\mathbf{u}_{\pi}$  is the piecewise polynomial of degree $k$ defined in Lemma \ref{lm:scott}. Then, from estimate~\eqref{eq:scott} and the previous estimate on the velocity error, we obtain
\begin{equation}
\label{eq:mu_2bis}
\begin{split}
|b(\mathbf{v}_h, p_h - p)| &\leq  C h^k \, |\mathbf{u}|_{k+1} \,  \|\mathbf{v}_h\|_{\mathbf{V}}.
\end{split}
\end{equation}
Moreover, we have
\begin{equation}
\label{eq:thm4.4}
|b(\mathbf{v}_h, p - q_h) | \leq C \|p - q_h\|_{Q}\|\mathbf{v}_h\|_{\mathbf{V}}.
\end{equation}
Then, using \eqref{eq:mu_2bis} and \eqref{eq:thm4.4} in \eqref{eq:thm4.1}, we infer

\begin{equation}\label{eq:press-est}
\|p_h - q_h\|_Q \leq C h^k \, |\mathbf{u}|_{k+1} + C  \|p - q_h\|_{Q}.
\end{equation}
Finally, using~\eqref{eq:press-est} and the triangular inequality, we get
\[
\|p -p_h\|_Q \leq \|p -q_h\|_Q + \|p_h - q_h\|_Q \leq C h^k \, |\mathbf{u}|_{k+1}   + C  \|p - q_h\|_{Q} \qquad \text{for all $q_h \in Q_h$.}
\]
Passing to the infimum with respect to $q_h \in Q_h$, and using estimate \eqref{eq: stime q_I}, we get the thesis.
\end{proof}

\begin{remark}
We observe that the estimates on the velocity errors in Theorem \ref{thm5} do not depend on the continuous pressure, whereas the velocity errors of the classical methods have a pressure contribution. Therefore the proposed scheme belongs to the class of the \textbf{pressure-robust methods}.
\end{remark}

% -----------------------------------------------------------------------------------------------------------------------------
% -----------------------------------------------------------------------------------------------------------------------------

% -----------------------------------------------------------------------------------------------------------------------------
% -----------------------------------------------------------------------------------------------------------------------------

\section{A Stable VEM for Brinkman Equations}
\label{sec:6}

\subsection{The continuous problem}
The Brinkman equation describes fluid flow in complex porous media with a viscosity coefficient highly varying so that the flow is dominated by the Darcy equations in some regions of the domain  and by the Stokes equation in others. We consider the Brinkman equation on a polygon $\Omega \subseteq \R^2$ with homogeneous Dirichlet boundary
conditions:
\begin{equation}
\label{eq:brinkman primale}
\left\{
\begin{aligned}
&  \mu \, \boldsymbol{\Delta}  \mathbf{u}  +  \nabla p  +  \K^{-1} \mathbf{u} = \mathbf{f} \qquad  & &\text{in $\Omega$,} \\
& {\rm div} \, \mathbf{u} = 0 \qquad & &\text{in $\Omega$,} \\
& \mathbf{u}  = \mathbf{0}  \qquad & &\text{on $\partial \Omega$,}
\end{aligned}
\right.
\end{equation}
where $\mathbf{u}$ and $p$ are the unknown velocity and  pressure fields, $\mu$ is the fluid viscosity, $\K$ denotes the permeability tensor of the porous media and $\mathbf{f} \in [L^2(\Omega)]^2$ is the external source term. 
We assume  that  $\K$ is a symmetric positive definite tensor and that there exist two positive (uniform) constants  $\lambda_1$, $\lambda_2 > 0$ such that
\begin{equation*}
%\label{eq:tensore}
\lambda_1 \, \eta^T \eta \leq \eta^T \K^{-1} \eta \leq  \lambda_2 \, \eta^T  \eta \qquad \text{for all $\eta \in \R^2$.}
\end{equation*}
For what concerns the fluid viscosity we consider $0 < \mu \leq C$, this include  the case where $\mu$ approaches zero and  equation \eqref{eq:brinkman primale} becomes a singular perturbation of the classic Darcy equations.
Let us consider the spaces
\begin{equation*}
%\label{eq:spazi continui brinkman}
\mathbf{V}:= [H_0^1(\Omega)]^2, \qquad Q:= L^2_0(\Omega) 
\end{equation*}
with the usual norms, and let  $A(\cdot, \cdot) \colon \mathbf{V} \times \mathbf{V} \to \R$ be the bilinear form defined by:
\begin{equation*}
%\label{eq:forma a brinkman}
A (\mathbf{u},  \mathbf{v}) :=  a^{\nabla} (\mathbf{u},   \mathbf{v}) + a(\mathbf{u},   \mathbf{v}) , \qquad \text{for all $\mathbf{u},  \mathbf{v} \in \mathbf{V}$}
\end{equation*}
where
\[
a^{\nabla} (\mathbf{u},   \mathbf{v}) := \int_{\Omega} \mu \, \boldsymbol{\nabla} \mathbf{u} : \boldsymbol{\nabla} \mathbf{v} \,{\rm d}x \qquad \text{for all $\mathbf{u},  \mathbf{v} \in \mathbf{V}$}
\]
and $a(\cdot, \cdot)$ is the bilinear form defined in \eqref{eq:forma a}.
Then the  variational formulation of Problem \eqref{eq:brinkman primale} is:
\begin{equation}
\label{eq:brinkman variazionale}
\left\{
\begin{aligned}
& \text{find $(\mathbf{u}, p) \in \mathbf{V} \times Q$, such that} \\
& A(\mathbf{u}, \mathbf{v}) + b(\mathbf{v}, p) = (\mathbf{f}, \mathbf{v})  \qquad & \text{for all $\mathbf{v} \in \mathbf{V}$,} \\
&  b(\mathbf{u}, q) = 0 \qquad & \text{for all $q \in Q$,}
\end{aligned}
\right.
\end{equation}
where and $b(\cdot, \cdot) \colon \mathbf{V} \times Q \to \R$  is the bilinear form in \eqref{eq:forma b} and using standard notation
\[
(\mathbf{f}, \mathbf{v})  = \int_{\Omega} \mathbf{f} \cdot \mathbf{v}\, {\rm d}x.
\] 
The natural energy norm for the velocities is induced by the symmetric an positive bilinear form $A(\cdot, \cdot)$ and is defined by (e.g. \cite{stenberg})
\begin{equation*}
%\label{eq:norma velocity brinkman}
\|\mathbf{v}\|^2_{\mathbf{V}, \mu} := A(\mathbf{v}, \mathbf{v}) = \mu \,\|\nabla \mathbf{v}\|_0^2 + \|\K^{-1/2} \mathbf{v}\|_0^2.
\end{equation*}
We can observe that the equivalence with the $\mathbf{V}$ norm  is not uniform, i.e.
\[
c_1  \sqrt{\mu}\, \|\mathbf{v}\|_{\mathbf{V}} \leq  \|\mathbf{v}\|_{\mathbf{V}, \mu} \leq c_2 \, \|\mathbf{v}\|_{\mathbf{V}}
\]
where $c_1$, $c_2$ here and in the follows denote two positive constant independent of $h$ and $\mu$. 
For what concerns the pressures,  we consider the  norm  (see for instance \cite{stenberg})
\begin{equation}
\label{eq:norma pressioni brinkman}
\|p\|_{Q, \mu} := \sup_{\mathbf{v} \in \mathbf{V}} \frac{b(\mathbf{v}, p)}{\|\mathbf{v}\|_{\mathbf{V}, \mu}}
\end{equation}
Using the inf-sup condition in the usual norm it is possible to check the equivalence between the norms for the pressure but again the equivalence is not uniform, i.e.
\[
c_1 \, \|p\|_Q \leq \|p\|_{Q, \mu} \leq \frac{c_2}{\sqrt{\mu}} \, \|p\|_{Q}.
\]
Since, considering the modified norm, the bilinear form $A(\cdot, \cdot)$ is  uniformly continuous and coercive, and inf-sup condition is clearly fulfilled, Problem \eqref{eq:brinkman variazionale} has a unique solution $(\mathbf{u}, p) \in \mathbf{V} \times Q$ such that
\[
\|\mathbf{u} \|_{\mathbf{V}, \mu} + \|p\|_{Q, \mu} \leq C \, \|\mathbf{f}\|_{\mathbf{V}'}
\]
where the constant $C$ depends only on $\Omega$.

%----------------------------------------------------------------------------------------------------------------------------------------
%----------------------------------------------------------------------------------------------------------------------------------------

\subsection{Virtual formulation for Brinkman equations}

Mathematically,  Brinkman equations can be viewed as a combination of  the Stokes and the
Darcy equation, that can change from place to place in the computational domain.
Therefore, numerical schemes for  Brinkman equations have to be
carefully designed to accommodate both  Stokes and Darcy simultaneously.
In this section we propose a Virtual Element scheme  that is 
accurate for both Darcy and Stokes flows.
For this goal we combine the ideas  developed in the previous sections with the argument in \cite{stokes}.

Let us consider the virtual spaces $\mathbf{V}_h$ and $Q_h$ (cfr. \eqref{eq:V_h} and \eqref{eq:Q_h}). As usual in the VEM framework we need to define a computable approximation of the continuous bilinear forms.  Using obvious notations we split the bilinear form  $A(\cdot, \cdot)$ as
\begin{equation*}
%\label{eq:forma a locali continue brinkman}
A(\mathbf{u},  \mathbf{v}) =: \sum_{K \in \mathcal{T}_h} A^K (\mathbf{u},  \mathbf{v})  = \sum_{K \in \mathcal{T}_h} \left( a^{\nabla, K} (\mathbf{u},  \mathbf{v}) + a^K (\mathbf{u},  \mathbf{v})\right) \qquad \text{for all $\mathbf{u},  \mathbf{v} \in \mathbf{V}$.}
\end{equation*}
We begin by observing that, from \cite{stokes} (in particular c.f. $(27)-(29)$) and from Section \ref{sub:3.2}, $A^K(\mathbf{q}_k, \mathbf{v})$ is computable on the basis of the DoFs $\mathbf{D_V}$ for all $\mathbf{q}_k \in [\Pk_k(K)]^2$ and for all $\mathbf{v} \in \mathbf{V}_h$. 
Starting from this observation we can approximate the continuous form $A^K(\cdot, \cdot)$ with the bilinear form 
$A_h^K(\cdot, \cdot) \colon \mathbf{V}_h^K \times \mathbf{V}_h^K \to \R$,
given by
\begin{equation*}
%\label{eq:A_h}
A_h^K(\mathbf{u}, \mathbf{v}) = a_h^{\nabla, K}(\mathbf{u}, \mathbf{v}) + a_h^K(\mathbf{u}, \mathbf{v}) \qquad \text{for all $\mathbf{u}$, $\mathbf{v} \in \mathbf{V}_h$}
\end{equation*}
where $a_h^{\nabla, K}(\cdot, \cdot)$ is the bilinear form defined in equation $(35)$ in \cite{stokes} and $a_h^K(\cdot, \cdot)$ is defined in \eqref{eq:a_h^K}.
It is clear that the bilinear form $A_h^K(\cdot, \cdot)$ satisfies the $k$-consistency and the stability properties.
%
%\begin{itemize}
%
%\item $\mathbf{k}$\textbf{-consistency}: for all $\mathbf{q}_k \in [\Pk_k(K)]^2$ and $\mathbf{v}_h \in \mathbf{V}_h^K$
%
%\begin{equation}\label{eq:consist brinkman}
%A_h^K(\mathbf{q}_k, \mathbf{v}_h) = A^K( \mathbf{q}_k, \mathbf{v}_h);
%\end{equation}
%
%\item \textbf{stability}:  there exist  two positive constants $\alpha_*$ and $\alpha^*$, independent of $h$ and $K$, such that, for all $\mathbf{v}_h \in \mathbf{V}_h^K$, it holds
%
%\begin{equation}\label{eq:stabk brinkman}
%\alpha_* A^K(\mathbf{v}_h, \mathbf{v}_h) \leq A_h^K(\mathbf{v}_h, \mathbf{v}_h) \leq \alpha^* A^K(\mathbf{v}_h, \mathbf{v}_h).
%\end{equation}
%
%\end{itemize}
%The bilinear form $A_h^K$ can be defined as
%\begin{equation}
%\label{eq:a_h^k brinkman} 
%A_h^K(\mathbf{u}, \mathbf{v}) = a_h^{\nabla, K}(\mathbf{u}, \mathbf{v}) + a_h^K(\mathbf{u}, \mathbf{v}) \qquad \text{for all $\mathbf{u}$, $\mathbf{v} \in \mathbf{V}_h$}
%\end{equation}
%
As usual we build the global approximated bilinear form $A_h(\cdot, \cdot) \colon \mathbf{V}_h \times \mathbf{V}_h \to \R$ by simply summing the local contributions.
For what concerns the bilinear form $b(\cdot, \cdot)$, as observed in Section \ref{sub:3.3}, it can be computed exactly.
The last step consists in constructing a computable approximation of the right-hand side $(\mathbf{f}, \, \mathbf{v})$ in \eqref{eq:brinkman variazionale}.  We define the approximated load term $\mathbf{f}_h$ as 
\begin{equation}
\label{eq:f_h}
\mathbf{f}_h := \Pi_{k}^{0,K} \mathbf{f} \qquad \text{for all $K \in \mathcal{T}_h$,}
\end{equation}
and consider:
\begin{equation}
\label{eq:right}
(\mathbf{f}_h, \mathbf{v}_h)  = \sum_{K \in \mathcal{T}_h} \int_K \mathbf{f}_h \cdot \mathbf{v}_h \, {\rm d}K = \sum_{K \in \mathcal{T}_h} \int_K \Pi_{k}^{0,K} \mathbf{f} \cdot \mathbf{v}_h \, {\rm d}K = \sum_{K \in \mathcal{T}_h} \int_K \mathbf{f} \cdot \Pi_{k}^{0,K}  \mathbf{v}_h \, {\rm d}K.
\end{equation}
We observe that \eqref{eq:right} can be exactly computed from $\mathbf{D_V}$ for all $\mathbf{v}_h \in \mathbf{V}_h$ (see Proposition \ref{prp:pkprojection}). 
Furthermore, the following result concerning a $L^2$ and  $H^1$-type norm, can be proved using standard arguments \cite{VEM-volley}.
\begin{lemma}
\label{lemma2}
Let $\mathbf{f}_h$ be defined as in \eqref{eq:f_h},  and let us assume $\mathbf{f} \in H^{k+1}(\Omega)$. Then, for all $\mathbf{v}_h \in \mathbf{V}_h$, it holds
\begin{gather*}
\left|( \mathbf{f}_h - \mathbf{f}, \mathbf{v}_h ) \right| \leq C h^{k+1} |\mathbf{f}|_{k+1} \|\mathbf{v}_h\|_{0} \qquad \text{and} \qquad
\left|( \mathbf{f}_h - \mathbf{f}, \mathbf{v}_h ) \right| \leq C h^{k+2} |\mathbf{f}|_{k+1} |\mathbf{v}_h|_{\mathbf{V}}.
\end{gather*}
\end{lemma}
In the light of the previous definitions, we consider the virtual element approximation of the Brinkman problem:
\begin{equation}
\label{eq:brinkman virtual}
\left \{
\begin{aligned}
& \text{find $(\mathbf{u}_h, p_h) \in \mathbf{V}_h \times Q_h$, such that} \\
& A_h(\mathbf{u}_h, \mathbf{v}_h) + b(\mathbf{v}_h, p_h) = (\mathbf{f}_h, \, \mathbf{v}_h)  \qquad  \text{for all $\mathbf{v}_h \in \mathbf{V}_h$,}\\
&  b(\mathbf{u}_h, q_h) = 0 \qquad  \text{for all $q_h \in Q_h$.}
\end{aligned}
\right.
\end{equation}
Equation \eqref{eq:brinkman virtual} is well posed since the discrete bilinear form $A_h(\cdot,\cdot)$ 
is (uniformly) stable with respect to the norm $\|\cdot\|_{\mathbf{V}, \mu}$ by construction and the inf-sup condition is fulfilled (the proof follows the guidelines of Proposition 4.2 in \cite{stokes} and the linearity of the Fortin operator). Then we have the following result.

\begin{theorem}
Problem \eqref{eq:brinkman virtual} has a unique solution $(\mathbf{u}_h, p_h) \in \mathbf{V}_h \times Q_h$, verifying the estimate
\[
\|\mathbf{u}_h\|_{\mathbf{V}, \mu} + \|p_h\|_{Q, \mu} \leq C \|\mathbf{f}\|_{\mathbf{V}'}.
\]
\end{theorem}
%----------------------------------------------------------------------------------------------------------------------------------------
%----------------------------------------------------------------------------------------------------------------------------------------
We now notice that, if $\mathbf{u} \in \mathbf{V}$ is the velocity solution to Problem \eqref{eq:brinkman variazionale}, then it is the solution
to Problem:
\begin{equation}
\label{eq:ridotto brinkman}
\left \{
\begin{aligned}
& \text{find $\mathbf{u} \in \mathbf{Z}$ such that}\\
& A(\mathbf{u}, \mathbf{v}) = (\mathbf{f}, \mathbf{v}) \qquad \text{for all $\mathbf{v} \in \mathbf{Z}$}
\end{aligned}
\right .
\end{equation}
Analogously, if $\mathbf{u}_h \in \mathbf{V}_h$ is the velocity solution to Problem \eqref{eq:brinkman virtual}, then it is the solution to
Problem:
\begin{equation}
\label{eq:ridotto brinkman virtual}
\left \{
\begin{aligned}
& \text{find $\mathbf{u}_h \in \mathbf{Z}_h$ such that} \\
& A_h(\mathbf{u}_h, \mathbf{v}_h) = (\mathbf{f}_h, \mathbf{v}_h) \qquad \text{for all $\mathbf{v}_h \in \mathbf{Z}_h$}
\end{aligned}
\right .
\end{equation}
For what concerns the convergence results  we state the following theorem. The proof can be derived by extending the techniques of the previous section and is therefore omitted.
\begin{theorem}
Let $\mathbf{u} \in \mathbf{Z}$ be the solution of problem \eqref{eq:ridotto brinkman} and $\mathbf{u}_h \in \mathbf{Z}_h$ be the solution of
problem \eqref{eq:ridotto brinkman virtual}. Then
\[
\|\mathbf{u} - \mathbf{u}_h\|_{\mathbf{V}, \mu} \leq  C (\sqrt{\mu}\, h^k + \|\K^{-1/2}\|_{\infty} h^{k+1}) |u|_{k+1} + C \,h^{k+1}  \, |f|_{k+1} 
\]
Let $(\mathbf{u}, p) \in \mathbf{V} \times Q$ be the solution of Problem \eqref{eq:brinkman variazionale} and $(\mathbf{u}_h, p_h) \in \mathbf{W}_h \in Q_h$ be the solution of Problem \eqref{eq:brinkman virtual}. Then it holds:
\[
\|p - p_h \|_{Q, \mu} \leq C\, \left( h^k \, |u|_{k+1} +  \frac{h^k}{\sqrt{\mu}} |p|_k + h^{k+1} \, |f|_{k+1} \right).
\]
The constants $C$ above are independent of $h$ and $\mu$. 
\end{theorem}
%----------------------------------------------------------------------------------------------------------------------------------------
%----------------------------------------------------------------------------------------------------------------------------------------

%\begin{remark} 
%\label{remark:reduced}
%Following the approach of Section 5 in \cite{stokes}, the proposed discrete problem is immediately equivalent to a \textbf{reduced problem} where the pressures are piecewise constants and the  velocity space does not have the degrees of freedom $\mathbf{D_V4}$. See \cite{preprintns} for a deeper presentation.
%\end{remark}

In the last part of this section we present a brief discussion about the construction of a \textbf{reduced} virtual element method for Brinkman equations equivalent to Problem~\eqref{eq:brinkman virtual} but involving significantly fewer degrees of freedom, especially for large $k$.
This construction essentially follows the guidelines of Section 5 in \cite{stokes} (where we refer the reader for a deeper presentation).
Let us define the original reduced local virtual spaces, for $k\geq 2$:
\begin{multline*}
%\label{eq:W_h^Kbis}
\widehat{\mathbf{W}}_h^K := \left\{  
\mathbf{v} \in [H^1(K)]^2 \quad \text{s.t} \quad \mathbf{v}_{|{\partial K}} \in [\B_k(\partial K)]^2,  \,
\biggl\{
\begin{aligned}
& -\boldsymbol{\Delta}    \mathbf{v}  -  \nabla s \in \mathcal{G}_{k-2}(K)^{\perp},  \\
& {\rm div} \, \mathbf{v} \in \Pk_{0}(K),
\end{aligned}
\biggr. \qquad \text{for some $s \in H^1(K)$}
\quad \right\}
\end{multline*}
As before we enlarge the virtual space $\widehat{\mathbf{W}}_h^K$ and we consider
\begin{multline*}
%\label{eq:U_h^K ridotto}
\widehat{\mathbf{U}}_h^K := \left\{  
\mathbf{v} \in [H^1(K)]^2 \quad \text{s.t} \quad \mathbf{v}_{|{\partial K}} \in [\B_k(\partial K)]^2,  \,
\biggl\{
\begin{aligned}
& - \boldsymbol{\Delta}    \mathbf{v}  -  \nabla s \in \mathcal{G}_{k}(K)^{\perp},  \\
& {\rm div} \, \mathbf{v} \in \Pk_{0}(K),
\end{aligned}
\biggr. \qquad \text{for some $s \in H^1(K)$}
\quad \right\}
\end{multline*}
Finally we define the \textbf{enhanced Virtual Element space}, the restriction $\widehat{\mathbf{V}}_h^K$ of $\widehat{\mathbf{U}}_h^K$ given by
\begin{equation*}
%\label{eq:V_h^K ridotto}
\widehat{\mathbf{V}}_h^K := \left\{ \mathbf{v} \in \widehat{\mathbf{U}}_h^K \quad \text{s.t.} \quad   \left(\mathbf{v} - \Pi^{\nabla,K}_k \mathbf{v}, \, \mathbf{g}_k^{\perp} \right)_{[L^2(K)]^2} = 0 \quad \text{for all $\mathbf{g}_k^{\perp} \in  \mathcal{G}_{k}(K)^{\perp}/\mathcal{G}_{k-2}(K)^{\perp}$} \right\} ,
\end{equation*}
where as before the symbol $\mathcal{G}_{k}(K)^{\perp}/\mathcal{G}_{k-2}(K)^{\perp}$ denotes the polynomials i $\mathcal{G}_{k}(K)^{\perp}$ that are $L^2-$orthogonal to all polynomials of $\mathcal{G}_{k-2}(K)^{\perp}$.
For the pressures we consider the reduced space
\begin{equation*}
%\label{eq:Q_h^Kbis}
\widehat{Q}_h^K :=  \Pk_{0}(K).
\end{equation*}
As sets of degrees of freedom for the reduced spaces, combining the argument in Section \ref{sec:3} and \cite{stokes} we may consider the following. 
For every function $\mathbf{v} \in \widehat{\mathbf{V}}_h^K$ we take the following linear operators $\mathbf{\widehat{D}_V}$, split into three subsets (see Figure \ref{fig:dofslocr}):
\begin{itemize}
\item $\mathbf{\widehat{D}_V1}$: the values of $\mathbf{v}$ at each vertex of the polygon $K$,
\item $\mathbf{\widehat{D}_V2}$: the values of $\mathbf{v}$ at $k-1$ distinct points of every edge $e \in \partial K$,
\item $\mathbf{\widehat{D}_V3}$: the moments  of $\mathbf{v}$
\[
\int_K \mathbf{v} \cdot \mathbf{g}_{k-2}^{\perp} \, {\rm d}K \qquad \text{for all $\mathbf{g}_{k-2}^{\perp} \in \mathcal{G}_{k-2}(K)^{\perp}$.}
\]
\end{itemize} 

\begin{figure}[!h]
\center{
\includegraphics[scale=0.20]{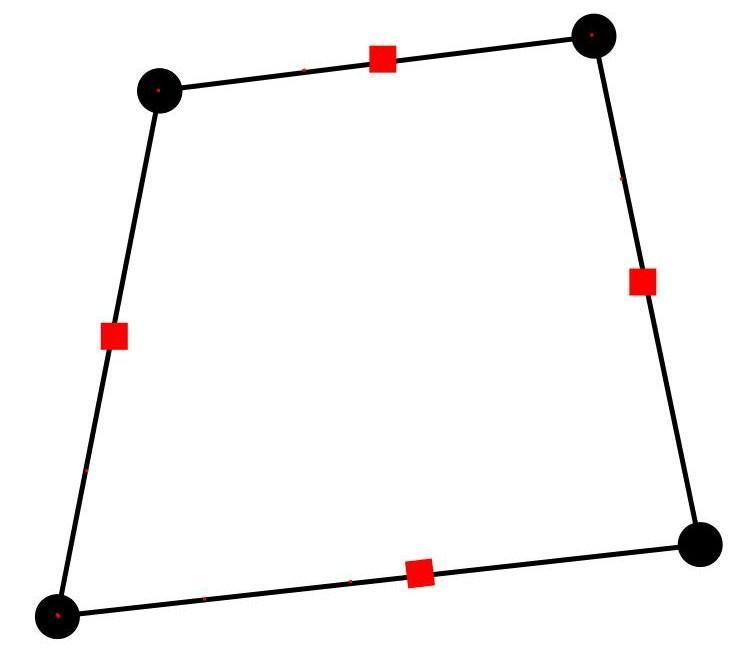} \qquad \qquad
\includegraphics[scale=0.20]{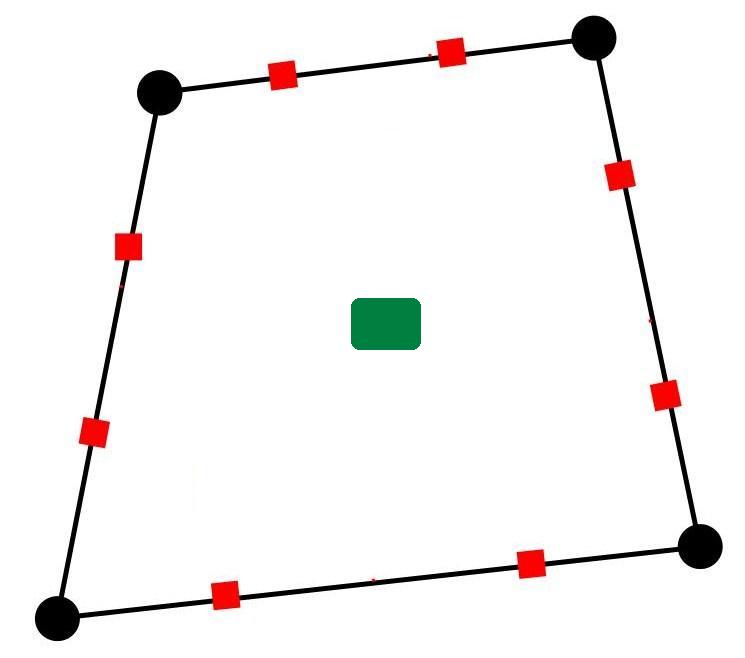}
\caption{Degrees of freedom for $k=2$, $k=3$. We denote $\mathbf{\widehat{D}_V1}$ with the black dots, $\mathbf{\widehat{D}_V2}$ with the red squares, $\mathbf{\widehat{D}_V3}$ with the green rectangles.}
\label{fig:dofslocr}
}
\end{figure}
For every $q \in \widehat{Q}_h$ we consider
\begin{itemize}
\item $\mathbf{\widehat{D}_Q}$: the moment
\[
\int_K q  \, {\rm d}K.
\]
\end{itemize}
Therefore we have that:
\begin{equation*}
%\label{eq:dimensione V_h^Kbis}
\dim\left( \widehat{\mathbf{V}}_h^K \right) = \dim\left([\B_k(\partial K)]^2\right) + \dim\left(\mathcal{G}_{k-2}(K)^{\perp}\right) = 2n_K k + \frac{(k-1)(k-2)}{2},
\end{equation*}
and
\begin{equation*}
%\label{eq:dimensione Q_h^Kbis}
\dim(\widehat{Q}_h^K) = \dim(\Pk_{0}(K))  = 1,
\end{equation*}
where $n_K$ is the number of vertexes in $K$.  

We  define the global reduced virtual element spaces in the standard fashion.
The reduced virtual element discretization of the Brinkman problem \eqref{eq:brinkman variazionale} is then:
\begin{equation}
\label{eq:brinkman reduced virtual}
\left\{
\begin{aligned}
& \text{find $\widehat{\mathbf{u}}_h \in \widehat{\mathbf{V}}_h$ and $\widehat{p}_h \in \widehat{Q}_h$, such that} \\
& A_h(\widehat{\mathbf{u}}_h, \widehat{\mathbf{v}}_h) + b(\widehat{\mathbf{v}}_h, \widehat{p}_h) = (\mathbf{f}_h, \widehat{\mathbf{v}}_h) \qquad & \text{for all $\widehat{\mathbf{v}}_h \in \widehat{\mathbf{V}}_h$,} \\
&  b(\widehat{\mathbf{u}}_h, \widehat{q}_h) = 0 \qquad & \text{for all $\widehat{q}_h \in \widehat{Q}_h$.}
\end{aligned}
\right.
\end{equation}
Above, the bilinear forms $A_h(\cdot, \cdot)$  and $b(\cdot, \cdot)$, and the loading term $ \mathbf{f}_h $  are the same as before. 
The following proposition states the relation between Problem~\eqref{eq:brinkman virtual} and the reduced Problem~\eqref{eq:brinkman reduced virtual} (the proof is equivalent to that of Proposition 5.1 in \cite{stokes}).
\begin{proposition}
\label{thm6}
Let $(\mathbf{u}_h, p_h) \in \mathbf{V}_h \times Q_h$ be the solution of problem \eqref{eq:brinkman virtual} and $(\widehat{\mathbf{u}}_h, \widehat{p}_h) \in \widehat{\mathbf{V}}_h \times \widehat{Q}_h$ be the solution of problem \eqref{eq:brinkman reduced virtual}. Then
\begin{equation*}
%\label{eq:equiv}
\widehat{\mathbf{u}}_h = \mathbf{u}_h   \qquad \text{and} \qquad \widehat{p}_{h|K} = \Pi_0^{0, K} p_h \quad \text{for all $K \in \mathcal{T}_h$.}
\end{equation*}
\end{proposition}

%\end{remark}

% -----------------------------------------------------------------------------------------------------------------------------
\section{Numerical tests}
\label{sec:7}
% -----------------------------------------------------------------------------------------------------------------------------

In this section we present two numerical experiments to test the practical performance of the method. The first experiment is focused on the method introduced in Section \ref{sec:3} for the Darcy Problem, whereas in the second experiment we test the method in Section \ref{sec:6} for the  Brinkman equations.

Since the VEM velocity solution $\mathbf{u}_h$ is not explicitly known point-wise inside the elements,  we  compute the method error comparing $\mathbf{u}$ with a suitable polynomial projection of the approximated $\mathbf{u}_h$. 
%
%To this end, for a given element $K\in\mathcal{T}_h$ and $k \geq 2$, we recall that the $L^2$ projection $\Pi_k^{0,K}$ from $\mathbf{V}_h$ (resp. in $\widetilde{\mathbf{V}}_h$) into the space of polynomials of degree less or equal that $k$ is exactly computable from the DoFs $\mathbf{D_V}$ (resp. $\mathbf{\widetilde{D}_V}$).
%Moreover using similar argument to that in Section \ref{sec:3} and Section \ref{sec:5}, we can see that
%\begin{itemize}
%\item the $L^2$ projection of ${\rm div} \mathbf{v}_h$ is computable from the DoFs $\mathbf{D_V}$ (resp. $\mathbf{\widetilde{D}_V}$) for all $\mathbf{v}_h \in \mathbf{V}_h$ (resp. in $\widetilde{\mathbf{V}}_h$), 
%\item introducing the tensor-valued $L^2$-projection operator  $\boldsymbol{\Pi}_{k-1}^{0, K} \colon [L^2(K)]^{2 \times 2} \to [\Pk_{k-1}(K)]^{2 \times 2}$, defined by
%\begin{equation}\label{projerror}
%\int_K \left(\mathbf{A} - \boldsymbol{\Pi}_{k-1}^{0, K} \mathbf{A} \right) \, : \, \mathbf{P}_{k-1}  \, {\rm d} x = 0 \qquad \text{for all $\mathbf{A}  \in [L^2(\Omega)]^{2 \times 2}$ and $\mathbf{P}_{k-1} \in [\Pk_{k-1}(K)]^{2 \times 2}$.}
%\end{equation}
%it is easy to derive that for every $\mathbf{v}_h$ in $\mathbf{V}_h$ (resp. in $\widetilde{\mathbf{V}}_h$), 
%$\boldsymbol{\Pi}_{k-1}^{0, K} \, \boldsymbol{\nabla} \mathbf{v}_h$ is exactly computable using the DoFs $\mathbf{D_V}$ (resp. $\mathbf{\widetilde{D}_V}$).
%\end{itemize}
In particular we consider the computable error quantities:
\begin{gather*}
{\rm error}(\mathbf{u}, H^1) := \left( \sum_{K \in \mathcal{T}_h} \left \| \boldsymbol{\nabla} \, u -  \boldsymbol{\Pi}_{k-1}^{0, K} (\boldsymbol{\nabla} \, u_h) \right \|_{0,K}^2 \right)^{1/2} \\
{\rm error}(\mathbf{u}, H({\rm div})) := \left( \sum_{K \in \mathcal{T}_h} \left \| {\rm div} \,u -  {\rm div} \,u_h \right \|_{0,K}^2  + \sum_{K \in \mathcal{T}_h} \left \| u -  \Pi_{k}^{0, K}  \,u_h \right \|_{0,K}^2 \right)^{1/2} \\
{\rm error}(\mathbf{u}, L^2)  := \left(\sum_{K \in \mathcal{T}_h} \left \| u -  \Pi_{k}^{0, K}  \,u_h \right \|_{0,K}^2 \right)^{1/2} \\
{\rm error}(p, L^2)  :=\|p - p_h\|_{0}.
\end{gather*}

Regarding the computational domain, in our tests we always take the square domain $\Omega= [0,1] ^2$, which is partitioned using the following sequences of polygonal meshes:
\begin{itemize}
\item $\{ \mathcal{V}_h\}_h$: sequence of Voronoi meshes with $h=1/4, 1/8, 1/16, 1/32$,
\item $\{ \mathcal{T}_h\}_h$: sequence of triangular meshes with $h=1/2, 1/4, 1/8, 1/16$,
\item $\{ \mathcal{Q}_h\}_h$: sequence of square meshes with $h=1/4, 1/8, 1/16, 1/32$.
\item $\{ \mathcal{W}_b\}_h$: sequence of WEB-like meshes with $h= 4/10, 2/10,  1/10, 1/20$.
\end{itemize}
An example of the adopted meshes is shown in Figure \ref{Figure1}.  
\begin{figure}[!h]
\centering
\includegraphics[scale=0.20]{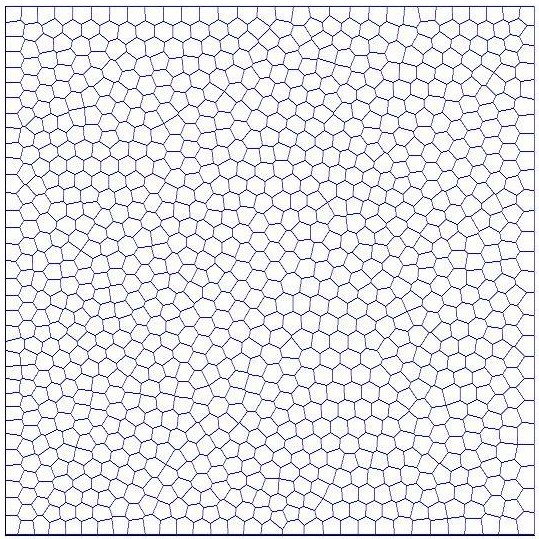} \quad
\includegraphics[scale=0.20]{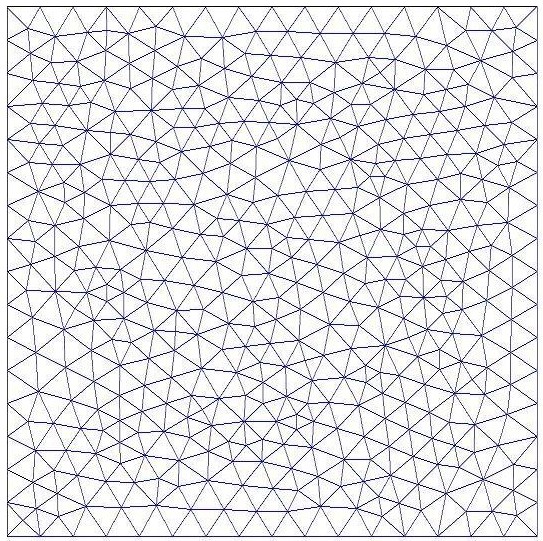} \quad
\includegraphics[scale=0.20]{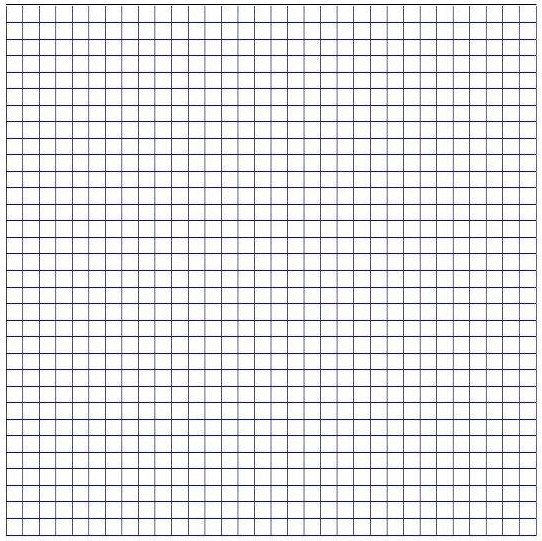} \quad
\includegraphics[scale=0.20]{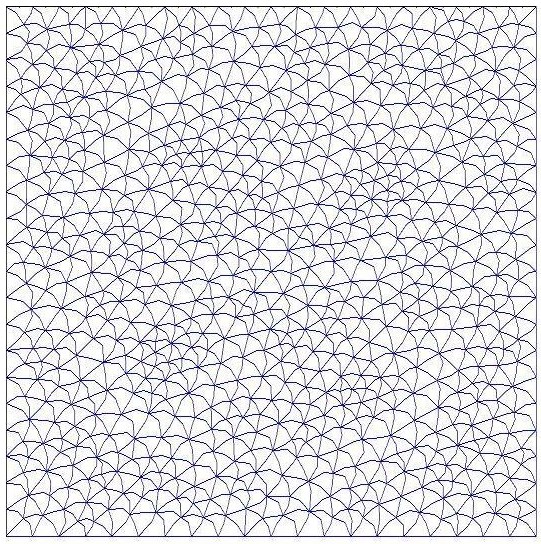} 
\caption{Example of polygonal meshes: $\mathcal{V}_{1/32}$,  $\mathcal{T}_{1/16}$, $\mathcal{Q}_{1/32}$, $\mathcal{W}_{1/20}$.}
\label{Figure1}
\end{figure}
For the generation of the Voronoi meshes we use the code Polymesher \cite{TPPM12}. 
The non convex WEB-like meshes are composed by hexagons, generated starting from the triangular meshes $\{ \mathcal{T}_h\}_h$ and randomly displacing the midpoint of each (non boundary) edge.

\begin{test}
\label{test1}
In this example we consider the Darcy problem \eqref{eq:darcy variazionale} where  we set $\K =I$, and we choose  the load term $\mathbf{f}$ in such a way that the analytical solution is 
\[
\mathbf{u}(x,y) =  -\pi \, \begin{pmatrix}
 \sin(\pi x) \cos(\pi y) \\
 \cos(\pi x) \sin(\pi y)
\end{pmatrix} \qquad 
p(x,y) = \cos(\pi x) \cos(\pi y).
\]
We analyse the practical performance of the virtual method by studying the errors versus the diameter $h$ of the meshes. In addition we compare the results obtained with the scheme of Section \ref{sec:3}, labeled as ``\text{div-free}'', with those obtained with the method in the Appendix, labeled as ``non div-free'' (in both cases we consider polynomial degrees $k=2$). 

\begin{remark}
\label{remark:comparison}
We notice that the "non div-free" method is a naive extension to the Darcy equation of the inf-sup stable scheme proposed in \cite{VEM-elasticity}. Since the scheme lacks a uniform ellipticity-on-the-kernel condition, it is not recommended for the problem under consideration. The purpose of the comparison is thus to underline the importance of the property $Z_h\subseteq Z$ (cf. Section 3.4) in the present context. 
\end{remark}

In Figure \ref{Figure2} and \ref{Figure3}, we display the results for the sequence of Voronoi    
meshes $\mathcal{V}_h$. In Figure \ref{Figure4} and \ref{Figure5}, we show the results for the sequence of meshes $\mathcal{T}_h$, while 
in Figure \ref{Figure6} and \ref{Figure7} we plot the results for the sequence of 
meshes $\mathcal{Q}_h$, finally in \ref{Figure8} and \ref{Figure9} we exhibit the results for the sequence of meshes $\mathcal{W}_h$.
 
\begin{figure}[!h]
\centering
\includegraphics[scale=0.3]{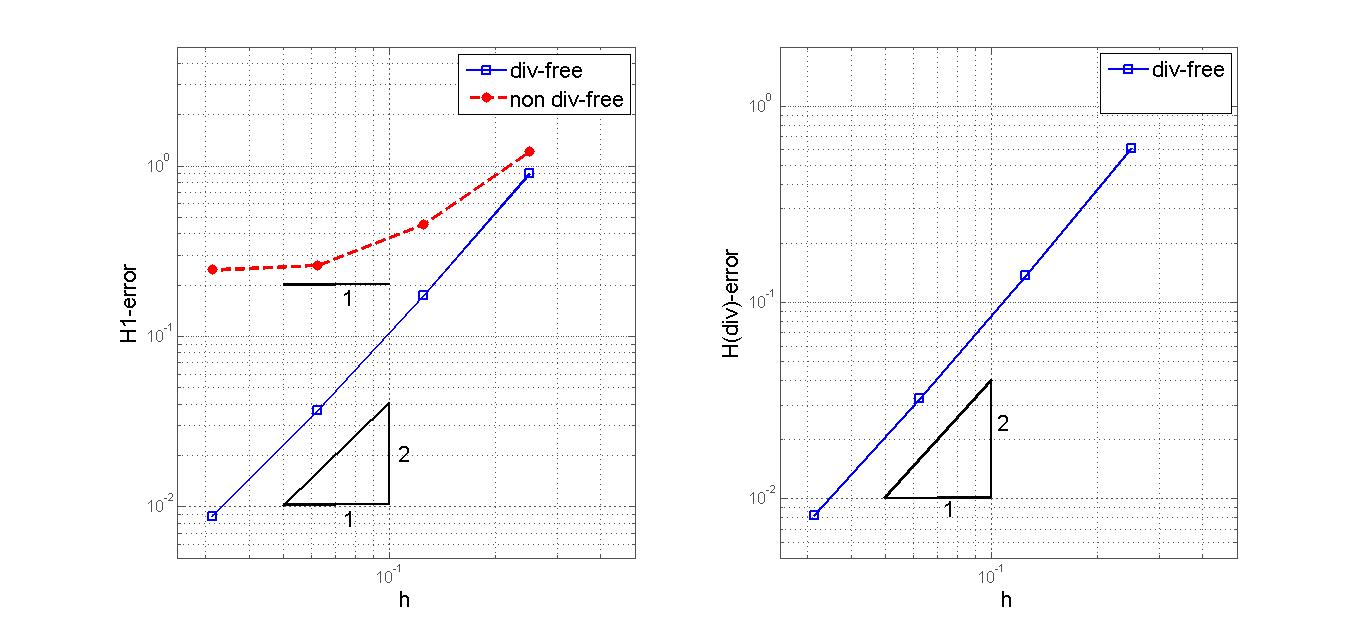}
\caption{Test \ref{test1}: behaviour of $H^1$ error and $H({\rm div})$ velocity error for the sequence of meshes $\mathcal{V}_h$ with $k=2$.}
\label{Figure2}
\end{figure}

\begin{figure}[!h]
\centering
\includegraphics[scale=0.3]{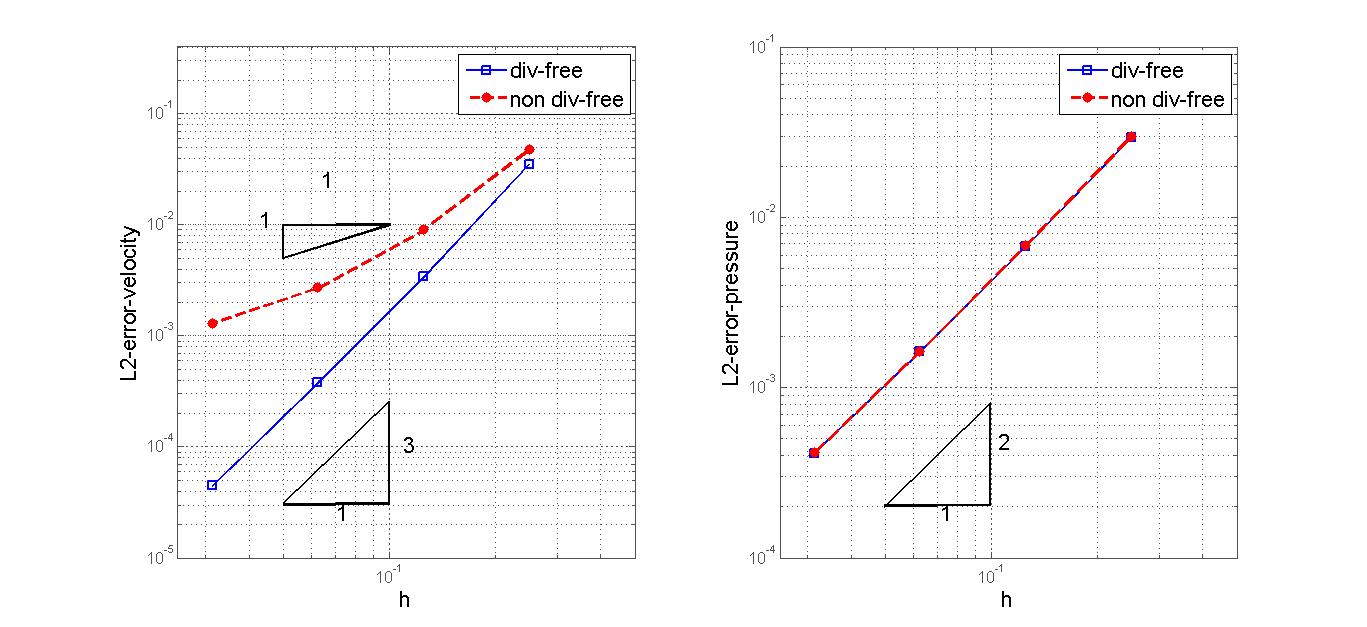}
\caption{Test \ref{test1}: behaviour of $L^2$ error both for the velocities and the pressures for the sequence of meshes $\mathcal{V}_h$ with $k=2$.}
\label{Figure3}
\end{figure}

\begin{figure}[!h]
\centering
\includegraphics[scale=0.3]{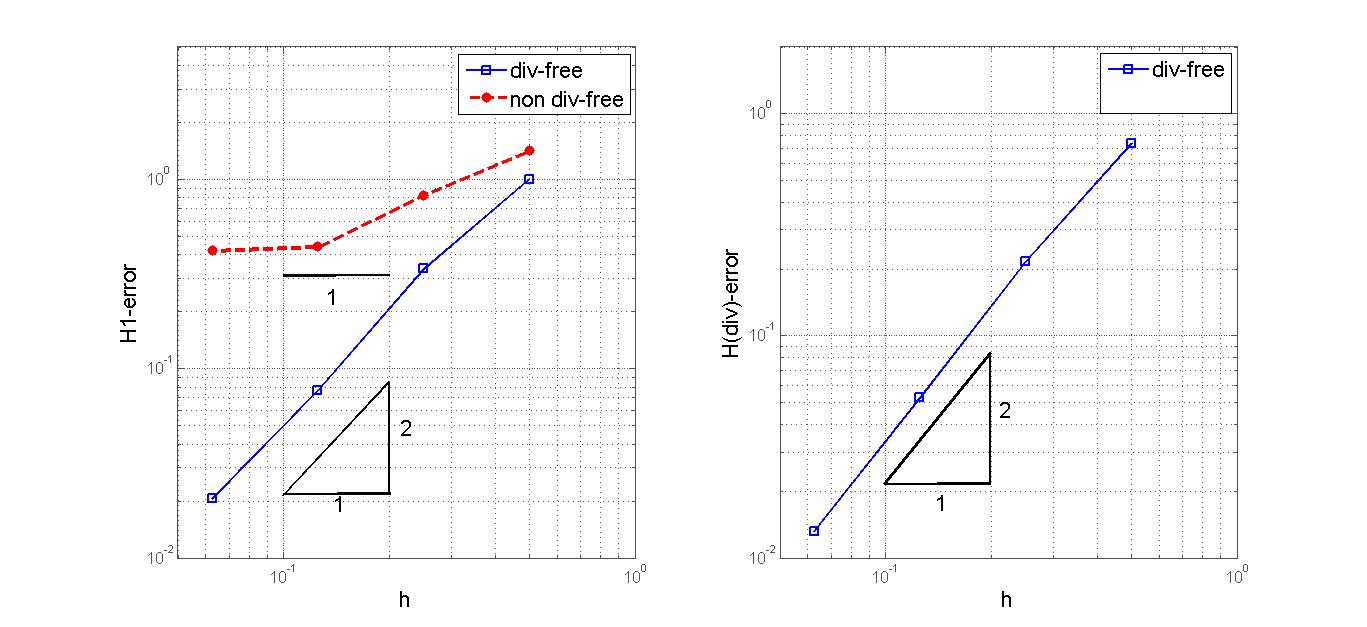}
\caption{Test \ref{test1}: behaviour of $H^1$ error and $H({\rm div})$ velocity  error for the sequence of meshes $\mathcal{T}_h$ with $k=2$.}
\label{Figure4}
\end{figure}

\begin{figure}[!h]
\centering
\includegraphics[scale=0.3]{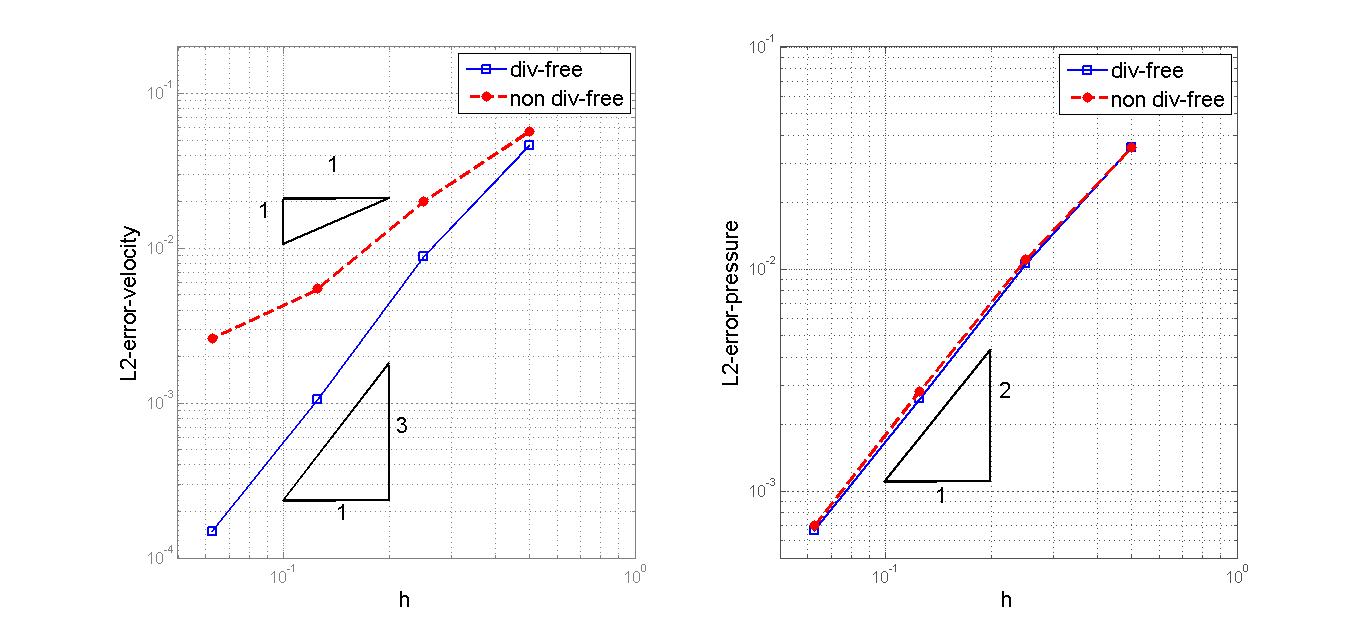}
\caption{Test \ref{test1}: behaviour of $L^2$ error both for the velocities and the pressures for the sequence of meshes $\mathcal{T}_h$ with $k=2$.}
\label{Figure5}
\end{figure}

\begin{figure}[!h]
\centering
\includegraphics[scale=0.3]{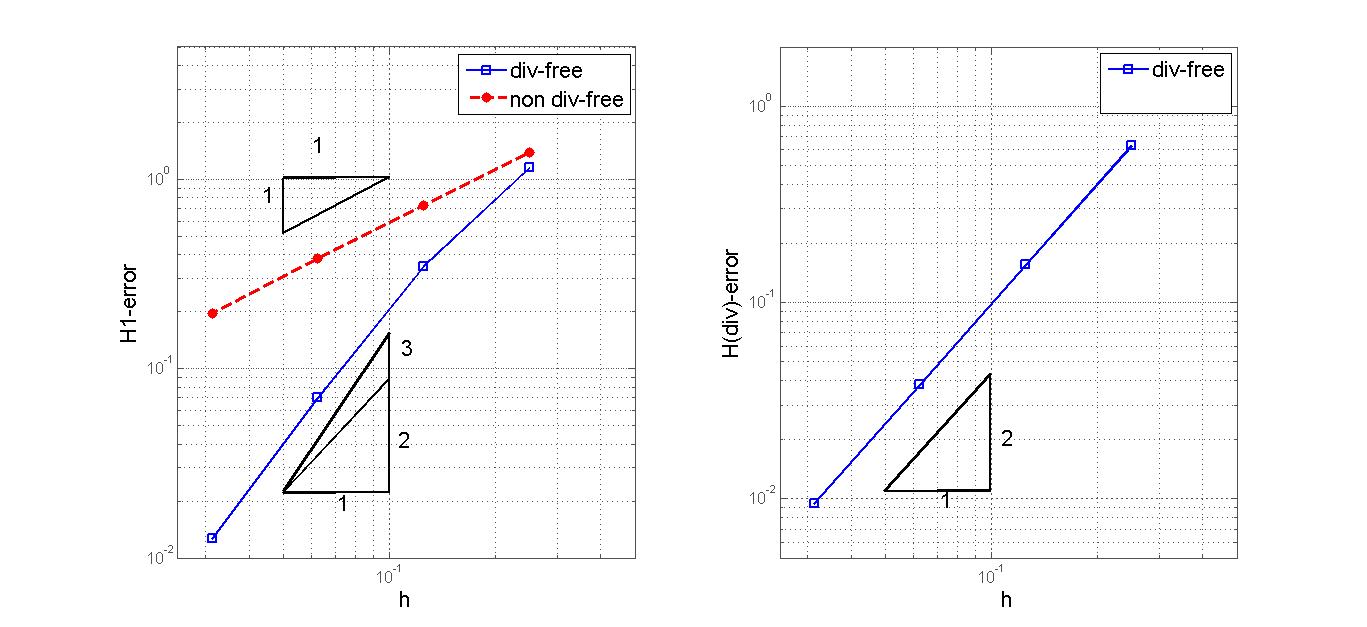}
\caption{Test \ref{test1}: behaviour of $H^1$ error and $H({\rm div})$ velocity error for the sequence of meshes $\mathcal{Q}_h$ with $k=2$.}
\label{Figure6}
\end{figure}

\begin{figure}[!h]
\centering
\includegraphics[scale=0.3]{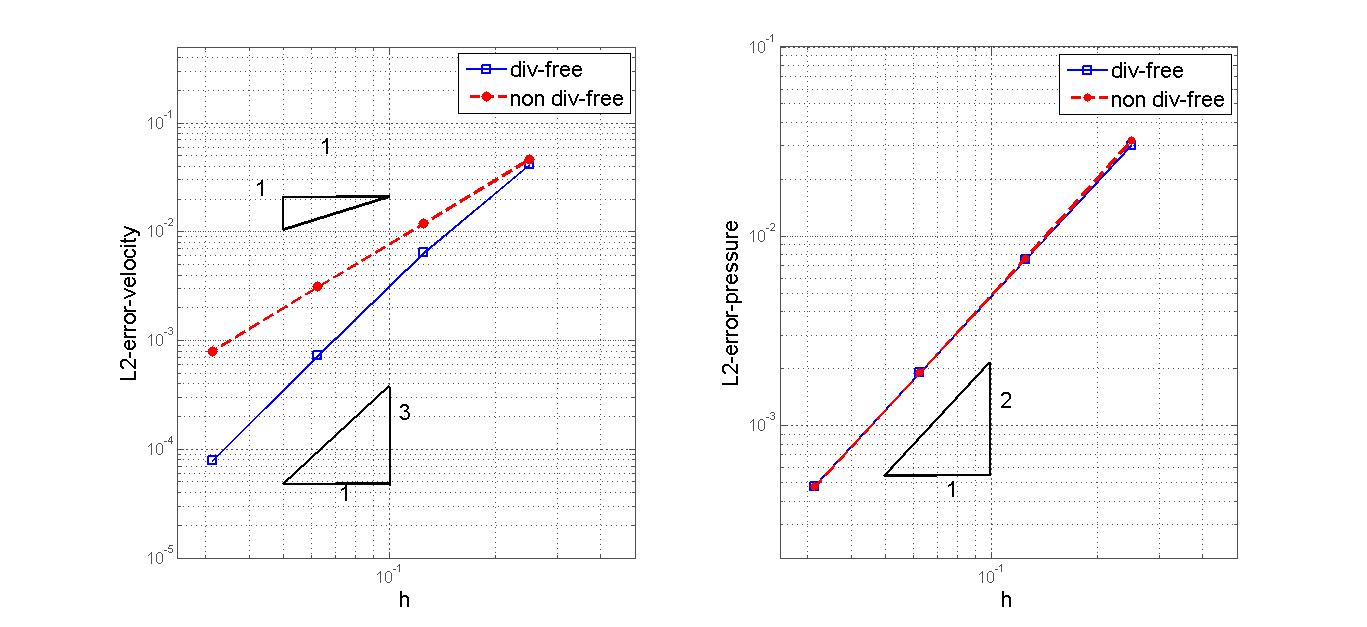}
\caption{Test \ref{test1}: behaviour of $L^2$ error both for the velocities and the pressures for the sequence of meshes $\mathcal{Q}_h$ with $k=2$.}
\label{Figure7}
\end{figure}

\begin{figure}[!h]
\centering
\includegraphics[scale=0.3]{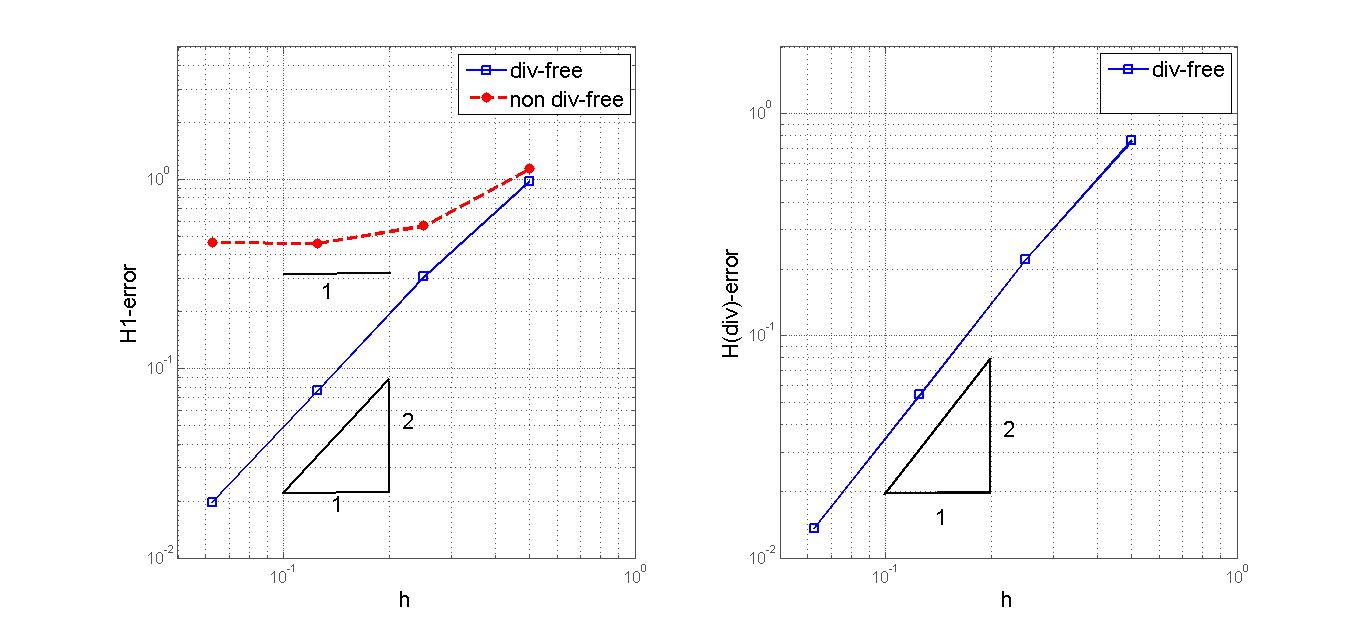}
\caption{Test \ref{test1}: behaviour of $H^1$ error and $H({\rm div})$ velocity error for the sequence of meshes $\mathcal{W}_h$ with $k=2$.}
\label{Figure8}
\end{figure}

\begin{figure}[!h]
\centering
\includegraphics[scale=0.3]{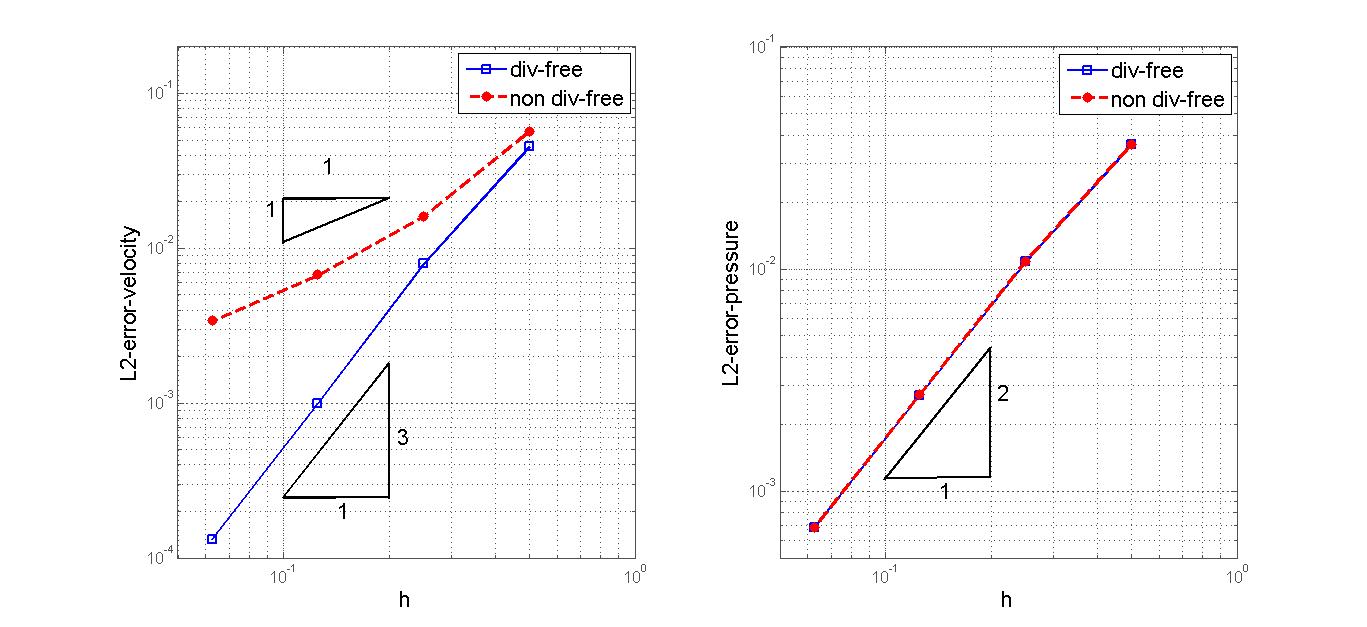}
\caption{Test \ref{test1}: behaviour of $L^2$ error both for the velocities and the pressures for the sequence of meshes $\mathcal{W}_h$ with $k=2$.}
\label{Figure9}
\end{figure}

We notice that the theoretical predictions of Section \ref{sec:4} and the Appendix are confirmed for both the $L^2$ norm and the $H({\rm div})$ norm. Note that for the $H({\rm div})$ norm we plot only the error for the ``div-free'' method since such scheme guarantees by construction, a better approximation of the divergence. Indeed  let $u_h$ (resp. $\widetilde{u}_h$) be the solution obtained with the ``div-free'' method (``non div-free'' method)  then $u_h$ satisfies
\[
 {\rm div} \, u_h =   \Pi^{0,K}_{k-1} f = \Pi^{0,K}_{k-1} ({\rm div} \, u) \qquad \text{for all $K \in \mathcal{T}_h$}
\]
whereas $\widetilde{u}_h$ satisfies the same equation only in a projected sense, i.e.
\[
 \Pi^{0,K}_{k-1}  ({\rm div} \, \widetilde{u}_h)  =  \Pi^{0,K}_{k-1}  f  = \Pi^{0,K}_{k-1}  ({\rm div} \,u)   \qquad \text{for all $K \in \mathcal{T}_h$.}
\]
We can observe that the convergence rate of the $L^2$ norm for the pressure is optimal also for the ``non div-free'' method as proved in the Appendix.
Finally, we can observe that using a square mesh decomposition holds a convergence rate that is slightly better than what predicted by the theory. 

\end{test}

\begin{test}
\label{test2}
In this example we test the Brinkman equation \eqref{eq:brinkman variazionale}  with different values of the fluid viscosity $\mu$ and fixed permeability tensor $\K = I$.  We choose the load term $\mathbf{f}$ and the Dirichlet boundary conditions in such a way that the analytical solution is  
\[
\mathbf{u}(x,y) =  \begin{pmatrix}
\sin(\pi x) \cos(\pi y)\\
-\cos(\pi x) \sin(\pi y)
\end{pmatrix} \qquad 
p(x,y) = x^2 y^2 - \frac{1}{9}.
\]
The aim of this test is to check the practical performance of the method introduced in Section \ref{sec:6} in the reduced formulation (c.f. \eqref{eq:brinkman reduced virtual}. 
In Table \ref{tabletest1} and Table \ref{tabletest2} we display the total amount  of DoFs and the errors for the family of meshes $\mathcal{V}_h$ choosing $k=2$ respectively for the ``div-free'' method (cf. Section \ref{sec:6}) and the ``non div-free'' method (cf. Reamrk \ref{remark:comparison} and the Appendix).
We observe that also in the limit case, when the equation becomes a singular perturbation of the classic Darcy equations (e.g. for ``small'' $\mu$), the proposed ``div-free'' method preserves the optimal order of accuracy.

\begin{table}[!h]
\centering
\begin{tabular}{ll*{3}{c}}
\toprule
&                         DoFs                      & ${\rm error}(\mathbf{u}, H^1)$   & ${\rm error}(\mathbf{u}, L^2)$       & ${\rm error}(p, L^2)$        \\
\midrule
\multirow{4}*{$\mu = 1e-01$}                             
& $182$       &$2.049871825e-01$                       &$9.414645391e-03$                             &$1.531569296e-02$  \\
& $702$       &$4.616835760e-02$                       &$8.379142208e-04$                             &  $2.796060666e-03$  \\
& $2794$      &$1.102679000e-02$                       &$9.416547836e-05$                             &  $5.322283997e-04$   \\
& $11210$     &$2.654465229e-03$                       &$1.104272204e-05$                             &  $1.261317758e-04$  \\
\midrule
\multirow{4}*{$\mu = 1e-04$}                        
& $182$       &$2.563406238e-01$                       &$1.296095515e-02$                              &  $6.431351247e-03$  \\
& $702$       &$5.462263791e-02$                       &$1.090600946e-03$                              &  $1.887150783e-03$  \\
& $2794$      &$1.246452741e-02$                       &$1.179870900e-04$                              &  $4.203480846e-04$   \\
& $11210$     &$2.790603844e-03$                       &$1.238676976e-05$                              &  $1.026912579e-04$  \\
\midrule
\multirow{4}*{$\mu = 1e-14$}                            
& $182$       &$2.572957705e-01$                       &$1.301886694e-02$                              &  $6.431351247e-03$  \\
& $702$       &$5.539413175e-02$                       &$1.111681710e-03$                              &  $1.887150783e-03$  \\
& $2794$      &$1.299961549e-02$                       &$1.253090639e-04$                              &  $4.203480846e-04$   \\
& $11210$     &$3.003059376e-03$                       &$1.394861018e-05$                              &  $1.026912579e-04$  \\
\bottomrule
\end{tabular}
\caption{Test \ref{test2}:  Error for the velocities and the pressures for the ``div-free'' method.}
\label{tabletest1}
\end{table}

\begin{table}[!h]
\centering
\begin{tabular}{ll*{3}{c}}  
\toprule
&                         DoFs                      & ${\rm error}(\mathbf{u}, H^1)$   & ${\rm error}(\mathbf{u}, L^2)$       & ${\rm error}(p, L^2)$        \\
\midrule
\multirow{4}*{$\mu = 1e-01$}                             
& $246$       &$2.074920846e-01$                       &$9.923688300e-03$                             &$1.597079423e-02$  \\
& $958$       &$4.660732867e-02$                       &$8.775839313e-04$                             &  $2.901160116e-03$  \\
& $3818$      &$1.113343018e-02$                       &$1.007780759e-04$                             &  $5.396527876e-04$   \\
& $15306$     &$2.677875898e-03$                       &$1.189832113e-05$                             &  $1.265284965e-04$  \\
\midrule
\multirow{4}*{$\mu = 1e-04$}                        
& $246$       &$2.477687988e-01$                       &$1.163768487e-02$                              &  $6.584759134e-03$  \\
& $958$       &$8.730792836e-02$                       &$1.749981712e-03$                              &  $1.922719665e-03$  \\
& $3818$      &$5.060500911e-02$                       &$5.351452940e-04$                              &  $4.258968745e-04$   \\
& $15306$     &$3.007229784e-02$                       &$1.597800852e-04$                              &  $1.036140919e-04$  \\
\midrule
\multirow{4}*{$\mu = 1e-14$}                            
& $246$       &$2.485968435e-01$                       &$1.168014371e-02$                              &  $6.581632957e-03$  \\
& $958$       &$9.149517231e-02$                       &$1.837032059e-03$                              &  $1.922719665e-03$  \\
& $3818$      &$6.422656370e-02$                       &$6.779428128e-04$                              &  $4.260009082e-04$   \\
& $15306$     &$6.289228510e-02$                       &$3.325972217e-04$                              &  $1.036819751e-04$  \\
\bottomrule
\end{tabular}
\caption{Test \ref{test2}:  Error for the velocities and the pressures for the ``non div-free'' method.}
\label{tabletest2}
\end{table}
\end{test}

\section{Acknowledgements}
The author wishes to thank L. Beir\~ao da Veiga and C. Lovadina for several interesting discussions and suggestions on the paper.
The author was partially supported by the European Research Council through
the H2020 Consolidator Grant (grant no. 681162) CAVE, Challenges and Advancements in Virtual Elements. This support is gratefully acknowledged.

%The author thanks the National Group of Scientific Computing (GNCS-INDAM) that supported this research through the project:  ``Finanziamento Giovani Ricercatori 2015-2016''.

% -------------------------------------------------------------------------------------------------
\section*{Appendix: Non divergence-free virtual space}\label{sec:5}
% -------------------------------------------------------------------------------------------------

We have built a new $H^1$-conforming (vector valued) virtual space for the velocity vector field different from the more standard one presented in \cite{VEM-elasticity} for the elasticity problem. 
The topic of the present section is to analyse the extension to the Darcy equation of the scheme of  \cite{VEM-elasticity}.
Even though the method should not be used for the Darcy problem (cf. Remark 6.1), the numerical experiments have shown an optimal error convergence rate for the pressure variable. In this Section, we theoretically explain such a behaviour, under a convexity assumption on $\Omega$ (essentially, a regularity assumption on the problem). To this end, we develop an inverse estimate for the VEM spaces which is interesting on its own, and can be used in other contexts.  
We briefly describe the method by making use of various tools from the Virtual Element technology, and  we  
refer the interested reader to the papers \cite{VEM-volley,VEM-enhanced,VEM-hitchhikers, VEM-elasticity}) for a deeper presentation.
We consider the local virtual space
\begin{equation*}
%\label{eq:W_h^Kclassic}
\widetilde{\mathbf{W}}_h^K := \left\{  
\mathbf{v} \in [H^1(K)]^2 \quad \text{s.t} \quad \mathbf{v}_{|{\partial K}} \in [\B_k(\partial K)]^2 \, , \quad
\boldsymbol{\Delta}    \mathbf{v}   \in [\Pk_{k-2}(K)]^2 \right\}
\end{equation*}
with local degrees of freedom $\mathbf{\widetilde{D}_V}$:
\begin{itemize}
\item $\mathbf{\widetilde{D}_V1}$: the values of $\mathbf{v}$ at each vertex of the polygon $K$,
\item $\mathbf{\widetilde{D}_V2}$: the values of $\mathbf{v}$ at $k-1$ distinct points of every edge $e \in \partial K$,
\item $\mathbf{\widetilde{D}_V3}$: the moments  of $\mathbf{v}$ up to order $k-2$, i.e.
\[
\int_K \mathbf{v} \, \cdot  \mathbf{q}_{k-2} \, {\rm d}K \qquad \text{for all $\mathbf{q}_{k-2} \in [\Pk_{k-2}(K)]^2$.}
\] 
\end{itemize} 
As observed in \cite{VEM-enhanced}, the DoFs $\mathbf{\widetilde{D}_V}$ allow us to compute the operator $\widetilde{{\Pi}}_k^{\nabla,K} \colon \widetilde{\mathbf{W}}_h^K \to [\Pk_k(K)]^2$ defined as the analogous of the $H^1$ semi-norm projection (c.f. \eqref{eq:Pi_k^K}).
%, defined by 
%\begin{equation}
%\label{eq:tildePi_k^K}
%\left\{
%\begin{aligned}
%& \int_K \boldsymbol{\nabla} \,\mathbf{q}_k : \boldsymbol{\nabla} (\mathbf{v}_h - \, \widetilde{{\Pi}}_k^{\nabla,K}  \mathbf{v}_h) \, {\rm d} K = 0 \qquad  \text{for all $\mathbf{q}_k \in [\Pk_k(K)]^2$,} \\
%& P^{0,K}(\mathbf{v}_h - \,  \widetilde{{\Pi}}_k^{\nabla,K}  \mathbf{v}_h) = \mathbf{0} \, ,
%\end{aligned}
%\right.
%\end{equation} 
For all $K \in\mathcal{T}_h$, the \textbf{augmented virtual local space} $\widetilde{\mathbf{U}}_h^K$ is defined by
\[
\widetilde{\mathbf{U}}_h^K = \left\{ \mathbf{v} \in [H^1(K)]^2  \quad \text{s.t.} \quad  \mathbf{v} \in [\B_k(\partial K)]^2, \, \boldsymbol{\Delta} \mathbf{v} \in [\Pk_{k}(K)]^2 \right\}.
\]
Now we define the \textbf{enhanced Virtual Element space}, the restriction $\mathbf{\widetilde{V}}_h^K$ of $\widetilde{\mathbf{U}}_h^K$ given by
\begin{equation*}
%\label{eq:localspace}
\widetilde{\mathbf{V}}_h^K := \left\{ \mathbf{v} \in \widetilde{\mathbf{U}}_h^K \quad \text{s.t.} \quad   \left(\mathbf{v} - \widetilde{{\Pi}}_k^{\nabla,K}\, \mathbf{v}, \, \mathbf{q}_k \right)_{[L^2(K)]^2} = 0 \quad \text{for all $\mathbf{q} \in [\Pk_{k}(K)/\Pk_{k-2} (K)]^2$} \right\} ,
\end{equation*}
where the symbol $\Pk_{k}(K)/\Pk_{k-2} (K)$ denotes the polynomials of degree $k$ living on $K$ that are $L^2-$orthogonal to all polynomials of degree $k-2$ on $K$.
The enhanced space $\widetilde{\mathbf{V}}_h^K$ has three fundamental properties (see \cite{VEM-enhanced} for a proof): 
\begin{itemize}
\item $[\Pk_k(K)]^2 \subseteq \widetilde{\mathbf{V}}_h^K$,
\item the set of linear operators $\mathbf{\widetilde{D}_V}$ constitutes a set of DoFs for the space $\widetilde{\mathbf{V}}_h^K$, 
\item the  $L^2$-projection operator $\widetilde{\Pi}^{0, K}_{k} \colon \widetilde{\mathbf{V}}_h^K \to [\Pk_k(K)]^2$ is exactly computable by the DoFs.
\end{itemize}
%Moreover it holds that we have that
%\[
%{\rm dim} \left(  \widetilde{\mathbf{V}}_h^K \right) = 2n_K k + (k-1)(k-2)   
%\]
%where $n_K$ is the number of vertexes of the polygon $K$. 
Recalling \eqref{eq:dimensione V_h^K} and from \cite{VEM-elasticity} it holds that ${\rm dim} \left(  \widetilde{\mathbf{V}}_h^K \right) = {\rm dim} \left(  \mathbf{V}_h^K \right)$. 
For the pressures we use the space of the piecewise polynomials $Q_h$ (c.f. \eqref{eq:Q_h}).
For what concerns the construction of the approximated bilinear forms, it is straightforward to see that
\begin{equation*}
%\label{eq:tildeb_h}
b(\mathbf{v}, q) = \sum_{K \in \mathcal{T}_h} b^K(\mathbf{v}, q) = \sum_{K \in \mathcal{T}_h}  \int_K {\rm div} \, \mathbf{v} \, q \,{\rm d}K = - \int_K \mathbf{v} \cdot \nabla q \,{\rm d}K  + \int_{\partial K} q \, \mathbf{v} \cdot \mathbf{n}
\end{equation*}
is computable from the DoFs for all $\mathbf{v} \in \widetilde{\mathbf{V}}_h$, $q \in Q_h$. Moreover using standard arguments \cite{VEM-enhanced, vaccabis} we can define a computable bilinear form 
\begin{equation*}
%\label{eq:tildea_h^K} 
\widetilde{a}_h^K(\cdot, \cdot) \colon \widetilde{\mathbf{V}}_h^K \times \widetilde{\mathbf{V}}_h^K \to \R
\end{equation*}
approximating the continuous form $a^K(\cdot, \cdot)$, and satisfying the $k$-consistency (c.f. \eqref{eq:consist}) and the stability properties (c.f. \eqref{eq:stabk}).  
Finally we define the global approximated bilinear form $\widetilde{a}_h(\cdot, \cdot) \colon\widetilde{\mathbf{V}}_h \times \widetilde{\mathbf{V}}_h \to \R$ by simply summing the local contributions.
By construction (see for instance \cite{VEM-enhanced}) the discrete bilinear form $\widetilde{a}_h(\cdot,\cdot)$ is 
(uniformly) stable with respect to the $L^2$ norm.
We are now ready to state the proposed discrete virtual element problem:
\begin{equation}
\label{eq:darcy classic virtual}
\left\{
\begin{aligned}
& \text{find $(\widetilde{\mathbf{u}}_h, \widetilde{p}_h) \in \widetilde{\mathbf{V}}_h \times Q_h$, such that} \\
& \widetilde{a}_h(\widetilde{\mathbf{u}}_h, \mathbf{v}_h) + b(\mathbf{v}_h, \widetilde{p}_h) = 0 \qquad & \text{for all $\mathbf{v}_h \in \widetilde{\mathbf{V}}_h$,} \\
&  b(\widetilde{\mathbf{u}}_h, q_h) = (f, q_h) \qquad & \text{for all $q_h \in Q_h$.}
\end{aligned}
\right.
\end{equation}

We shall first prove an inverse inequality for the virtual element functions in $\widetilde{\mathbf{V}}_h$.
\begin{lemma}
\label{lm:inverse estimate}
Under the assumption $\mathbf{(A1)}$, $\mathbf{(A2)}$,  let $K \in \mathcal{T}_h$ and let $\mathbf{v}_h \in \widetilde{\mathbf{V}}_h^K$. Then the following \textbf{inverse estimate} holds
\begin{equation}
\label{eq:inv}
|\mathbf{v}_h|_{1,k} \leq c_{inv} \, h_K^{-1} \, \|\mathbf{v}_h\|_{0,E}
\end{equation}
where the constant $c_{inv}$ is independent of $\mathbf{v}_h$, $h_K$ and $K$.
\end{lemma}

\begin{proof}
We only sketch the proof, since we follow the guidelines of Lemma 3.1 and 3.3 in \cite{2016stability}. Let $\mathbf{v}_h \in \widetilde{\mathbf{V}}_h^K$, then
\begin{equation}
\label{eq:inv1}
|\mathbf{v}_h|^2_{1,K} = \int_K \boldsymbol{\nabla}\mathbf{v}_h \cdot \boldsymbol{\nabla}\mathbf{v}_h = - \int_K  \boldsymbol{\Delta}\mathbf{v}_h \,\mathbf{v}_h + \int_{\partial K} \mathbf{v}_h \, \boldsymbol{\nabla}\mathbf{v}_h \cdot \mathbf{n}_K.   
\end{equation}
Under the assumption $\mathbf{(A1)}$, $\mathbf{(A2)}$ and by Lemma 3.3 in \cite{2016stability} we get
\begin{equation}
\label{eq:inv2}
- \int_K  \boldsymbol{\Delta}\mathbf{v}_h \,\mathbf{v}_h  \leq \|\boldsymbol{\Delta}\mathbf{v}_h\|_{0,E} \|\mathbf{v}_h\|_{0,E} \leq C_1 \, h_K^{-1} \, |\mathbf{v}_h|_{1,E}\|\mathbf{v}_h\|_{0,E} 
\end{equation}
where the constant $C_1$ is independent of $\mathbf{v}_h$, $h_K$ and $K$. For what concerns the second addend in the right side of \eqref{eq:inv1}, under the assumptions $\mathbf{(A1)}$, $\mathbf{(A2)}$, and using Lemma 3.1 in \cite{2016stability}, for all $\mathbf{w} \in [H^{1/2}(\partial K)]^2$ the following holds: there exists an extension $\widetilde{\mathbf{w}} \in [H^1(K)]^2$ of $\mathbf{w}$ such that
\begin{equation}
\label{eq:inv4}
h_K^{-1} \, \|\widetilde{\mathbf{w}}\|_{0, K} + |\widetilde{\mathbf{w}}|_{1, K} \leq C \,\|\mathbf{w}\|_{1/2, \partial K},
\end{equation}
where we consider the scaled norm
\begin{equation}
\label{eq:inv3}
\|\mathbf{w}\|_{1/2, \partial K} := h_K^{-1/2} \, \|\mathbf{w}\|_{0, \partial K} + |\mathbf{w}|_{1/2, \partial K}.
\end{equation}
By definition it holds
\[
\int_{\partial K} \mathbf{v}_h \, \boldsymbol{\nabla}\mathbf{v}_h \cdot \mathbf{n}_K \leq  \|\mathbf{v}_h\|_{1/2, \partial K} \, \sup_{\mathbf{w} \in [H^{1/2}(\partial K)]^2} \frac{\langle \boldsymbol{\nabla}\mathbf{v}_h \cdot \mathbf{n}_K , \mathbf{w}\rangle}{\|\mathbf{w}\|_{1/2, \partial K}}.
\]
Now, using the definition \eqref{eq:inv3}, an inverse estimate  ($\mathbf{v}_h$ is polynomial on $\partial K$) and the trace theorem \cite{MR2373954}, it holds that
\begin{equation}
\label{eq:inv5}
\begin{split}
\|\mathbf{v}_h\|_{1/2, \partial K} &= h_K^{-1/2} \, \|\mathbf{v}_h\|_{0, \partial K} + |\mathbf{v}_h|_{1/2, \partial K} \leq C_2 \, h_K^{-1/2} \, \|\mathbf{v}_h\|_{0, \partial K} \\
& \leq C_2 \,h_K^{-1/2} \, ( \, \|\mathbf{v}_h\|_{0, K})^{1/2}(h_K^{-1} \, \|\mathbf{v}_h\|_{0, K} +  |\mathbf{v}_h|_{1, K})^{1/2}
\\
& \leq \left(\epsilon + \frac{C_2^2}{\epsilon} \right)  h_K^{-1} \|\mathbf{v}_h\|_{0, K} + \epsilon |\mathbf{v}_h|_{1, K}.
\end{split}
\end{equation}
for any real $\epsilon >0$. For the last term, using  \eqref{eq:inv3}, \eqref{eq:inv4} and \eqref{eq:inv2} we get
\begin{equation}
\label{eq:inv6}
\begin{split}
\sup_{\mathbf{w} \in [H^{1/2}(\partial K)]^2} \frac{\langle \boldsymbol{\nabla}\mathbf{v}_h \cdot \mathbf{n}_K , \mathbf{w}\rangle}{\|\mathbf{w}_h\|_{1/2, \partial K}} &  \leq C \sup_{\widetilde{\mathbf{w}} \in [H^{1}(K)]^2} \frac{\langle \boldsymbol{\nabla}\mathbf{v}_h \cdot \mathbf{n}_K , \widetilde{\mathbf{w}}\rangle}{h_K^{-1} \, \|\widetilde{\mathbf{w}}\|_{0, K} + |\widetilde{\mathbf{w}}|_{1, K}} \\
& \leq C\,  \left( \sup_{\widetilde{\mathbf{w}} \in [H^{1}(K)]^2} \frac{\int_K \boldsymbol{\Delta}\mathbf{v}_h \,\widetilde{\mathbf{w}} }{h_K^{-1} \, \|\widetilde{\mathbf{w}}\|_{0, K} + |\widetilde{\mathbf{w}}|_{1, K}} + 
\sup_{\widetilde{\mathbf{w}}\in [H^{1}(K)]^2} \frac{\int_K \boldsymbol{\nabla}\mathbf{v}_h \cdot \boldsymbol{\nabla}\widetilde{\mathbf{w}}}{h_K^{-1} \, \|\widetilde{\mathbf{w}}\|_{0, K} + |\widetilde{\mathbf{w}}|_{1, K}} \right) \\
& \leq C \,  \left( \sup_{\widetilde{\mathbf{w}} \in [H^{1}(K)]^2} \frac{\int_K \boldsymbol{\Delta}\mathbf{v}_h \,\widetilde{\mathbf{w}}  }{h_K^{-1} \, \|\widetilde{\mathbf{w}}\|_{0, K} } + 
\sup_{\widetilde{\mathbf{w}} \in [H^{1}(K)]^2} \frac{\int_K \boldsymbol{\nabla}\mathbf{v}_h \cdot \boldsymbol{\nabla}\widetilde{\mathbf{w}}}{ |\widetilde{\mathbf{w}}|_{1, K}} \right) \\
& \leq C \, \left( h_K \, \|\boldsymbol{\Delta}\mathbf{v}_h\|_{0, K} + |\mathbf{v}_h|_{1,K} \right) \leq C_3 \, |\mathbf{v}_h|_{1,K}.
\end{split}
\end{equation}
From \eqref{eq:inv5} and \eqref{eq:inv6} we can conclude that
\begin{equation}
\label{eq:inv7}
\int_{\partial K} \mathbf{v}_h \, \boldsymbol{\nabla}\mathbf{v}_h \cdot \mathbf{n}_K \leq  \left( \left(\epsilon + \frac{C_2^2}{\epsilon} \right)  h_K^{-1} \|\mathbf{v}_h\|_{0, K} + \epsilon |\mathbf{v}_h|_{1, K} \right) \, C_3|\mathbf{v}_h|_{1, K}
\end{equation}
Finally, choosing $\epsilon = \frac{1}{2C_3}$ and collecting \eqref{eq:inv2} and \eqref{eq:inv7} in \eqref{eq:inv1} we have
\[
\frac{1}{2}|\mathbf{v}_h|_{1, K} \leq \left( C_1 + \frac{1}{2} + 2 C_2^2 C_3^2\right) \, h_K^{-1} \|\mathbf{v}_h\|_{0, K}
\]
from which follows the thesis.
\end{proof}

Let us analyse the theoretical properties of the method. We consider the discrete kernel:
\begin{equation*}
%\label{eq:tildeZ_h}
\widetilde{\mathbf{Z}}_h := \{ \mathbf{v}_h \in \widetilde{\mathbf{V}}_h \quad \text{s.t.} \quad b(\mathbf{v}_h, q_h) = 0 \quad \text{for all $q_h \in Q_h$}\} = \{ \mathbf{v}_h \in \widetilde{\mathbf{V}}_h \quad \text{s.t.} \quad \Pi_{k-1}^{0,K} ({\rm div} \mathbf{v}_h) = 0 \quad \text{for all $K \in \mathcal{T}_h$}\},
\end{equation*}
therefore the divergence-free property is satisfied only in a relaxed (projected) sense. As a consequence the bilinear form $\widetilde{a}_h(\cdot, \cdot)$ is \textbf{not uniformly coercive} on the discrete kernel $\widetilde{\mathbf{Z}}_h$; nevertheless it holds the following $h$-dependent coercivity property
% infect by using inverse estimate \cite{MR2373954}
\begin{equation}
\label{eq:noncorcivity}
\widetilde{a}_h(\mathbf{v}_h, \,\mathbf{v}_h) \geq \alpha \, \alpha_* \, \|\mathbf{v}_h\|_0^2 \geq  C \,  h^2 \, \|\mathbf{v}_h\|^2_{\mathbf{V}}
\end{equation}
that can be derived by using inverse estimate \eqref{eq:inv}.
Recalling that $\widetilde{a}_h(\cdot, \cdot)$ is continuous with respect the  ${\mathbf{V}}$ norm and that the discrete inf-sup condition is fulfilled  \cite{VEM-elasticity}
\begin{equation}
\label{eq:inf-sup tilde}
\sup_{\mathbf{v}_h \in \widetilde{\mathbf{V}}_h \, \mathbf{v}_h \neq \mathbf{0}} \frac{b(\mathbf{v}_h, q_h)}{ \|\mathbf{v}_h\|_{\mathbf{V}}} \geq \tilde{\beta} \|q_h\|_Q \qquad \text{for all $q_h \in Q_h$}
\end{equation}
problem \eqref{eq:darcy classic virtual} has a unique solution but we expect a worse order of accuracy since the bilinear form  $\widetilde{a}_h(\cdot, \cdot)$ is not uniformly stable. In fact we have the following convergence results that are, perhaps surprisingly, still optimal in the pressure variable.

\begin{theorem}
\label{thmclassic}
Let $(\mathbf{u}, p) \in \mathbf{V} \times Q$ be the solution of problem \eqref{eq:darcy variazionale} and $(\widetilde{\mathbf{u}}_h, \widetilde{p}_h) \in \widetilde{\mathbf{V}}_h \times Q_h$ be the solution of problem \eqref{eq:darcy classic virtual}. Then
\[
\begin{gathered}
\|\mathbf{u} -\widetilde{\mathbf{u}}_h \|_{0} \leq C \, h^{k-1} (|p|_{k} + h^2 |u|_{k+1}) , \qquad \text{and} \qquad \|\mathbf{u} - \widetilde{\mathbf{u}}_h \|_{\mathbf{V}} \leq C \,h^{k-2} (|p|_{k} + h^2 |u|_{k+1}).
\end{gathered}
\]
Assuming further that $\Omega$ is convex, the following estimate holds: 
\[
\|p - \widetilde{p}_h\|_{Q} \leq C \, h^{k} (|p|_{k} + h^2 |u|_{k+1}).
\]  
\end{theorem}

\begin{proof}
As observed in the proof of Theorem \ref{thm5} the inf-sup condition \eqref{eq:inf-sup tilde} implies the existence of a function $\widetilde{\mathbf{u}}_I \in  \widetilde{\mathbf{V}}_h$ such that 
\begin{gather}
\label{eq:divinterpolantbis}
\Pi_{k-1}^{0,K} ({\rm div} \, \widetilde{\mathbf{u}}_I) = \Pi_{k-1}^{0,K}({\rm div} \, \mathbf{u}) \qquad \text{for all $K \in \mathcal{T}_h$,} \\
\label{eq:stabilitybis}
\|\mathbf{u} - \widetilde{\mathbf{u}}_I\|_{\mathbf{V}} \leq C \, \inf_{\mathbf{v}_h \in \widetilde{\mathbf{V}}_h}\|\mathbf{u} - \mathbf{v}_h\|_{\mathbf{V}}
\end{gather}
Now let us set $\boldsymbol{\delta}_h = \widetilde{\mathbf{u}}_h - \widetilde{\mathbf{u}}_I$. For what concerns the $L^2$ norm, using the stability of the bilinear form $\widetilde{a}_h(\cdot, \cdot)$ and \eqref{eq:darcy classic virtual} together with \eqref{eq:darcy variazionale}
\begin{equation}
\label{eq:mutilde}
\begin{split}
\alpha_* \, \alpha \, \|\boldsymbol{\delta}_h\|^2_0 & \leq \alpha_* \, a(\boldsymbol{\delta}_h, \, \boldsymbol{\delta}_h) \leq \widetilde{a}_h(\boldsymbol{\delta}_h, \boldsymbol{\delta}_h) = \widetilde{a}_h(\widetilde{\mathbf{u}}_h, \boldsymbol{\delta}_h) - \widetilde{a}_h(\widetilde{\mathbf{u}}_I, \boldsymbol{\delta}_h) \\
& = - b(\boldsymbol{\delta}_h, \widetilde{p}_h) - a(\mathbf{u}, \boldsymbol{\delta}_h) + a(\mathbf{u}, \boldsymbol{\delta}_h) - \widetilde{a}_h(\widetilde{\mathbf{u}}_I, \boldsymbol{\delta}_h)  \\
& =   b(\boldsymbol{\delta}_h, p - \widetilde{p}_h) + ( a(\mathbf{u}, \boldsymbol{\delta}_h) - \widetilde{a}_h(\widetilde{\mathbf{u}}_I, \boldsymbol{\delta}_h)) =:
\mu_1(\boldsymbol{\delta}_h) + \mu_2(\boldsymbol{\delta}_h).
\end{split}
\end{equation}
By \eqref{eq:darcy classic virtual} and property \eqref{eq:divinterpolantbis}, it is straightforward to see that
\[
\Pi_{k-1}^{0,K} ({\rm div} \, \widetilde{\mathbf{u}}_h) =  \Pi_{k-1}^{0,K} ({\rm div} \, \mathbf{u}) = \Pi_{k-1}^{0,K} ({\rm div} \, \widetilde{\mathbf{u}}_I) \qquad \text{for all $K \in \mathcal{T}_h$} 
\]
so that $\boldsymbol{\delta}_h \in \widetilde{\mathbf{Z}}_h$. Therefore
\[
\mu_1(\boldsymbol{\delta}_h) = b(\boldsymbol{\delta}_h, p - \widetilde{p}_h)=  b(\boldsymbol{\delta}_h, p) = b(\boldsymbol{\delta}_h, p - q_h)  
\]
for all $q_h \in Q_h$. Using the inverse estimate \eqref{eq:inv} and standard approximation theory we get
\begin{equation}
\label{eq:mu1tilde}
|\mu_1(\boldsymbol{\delta}_h)| \leq C \|\boldsymbol{\delta}_h\|_{\mathbf{V}} \inf_{q_h \in Q_h} \|p - q_h\|_Q \leq C\, h^{-1} \,\|\boldsymbol{\delta}_h\|_0 \,h^{k}\, |p|_{k} = C \,h^{k-1} \,|p|_{k} \,  \|\boldsymbol{\delta}_h\|_0.
\end{equation}
By standard technique in VEM convergence theory, it holds that 
\begin{equation}
\label{eq:mu2tilde}
|\mu_2(\boldsymbol{\delta}_h)| \leq h^{k+1} \, |u|_{k+1} \, \|\boldsymbol{\delta}_h\|_0.
\end{equation}
Collecting  \eqref{eq:mu1tilde} and \eqref{eq:mu2tilde}  in \eqref{eq:mutilde} we get the $L^2$ estimate. Whereas the $\mathbf{V}$ norm estimate follows from an inverse estimate \eqref{eq:inv}.

For what concerns the estimate on the pressure, let  $p_{\pi}$ the piecewise polynomial with respect to $\mathcal{T}_h$ defined by  $p_{\pi}= \Pi_{k-1}^{0,K} \, p$ for all $K \in \mathcal{T}_h$. Let us set
\[
\boldsymbol{\chi}_h := \mathbf{u} - \widetilde{\mathbf{u}}_h, \qquad z:= p_{\pi} - p \qquad \rho_h:= p_{\pi} - \widetilde{p}_h 
\] 
From \eqref{eq:darcy variazionale} and \eqref{eq:darcy classic virtual}, it is straightforward to see that the couple $(\boldsymbol{\chi}_h, \rho_h)$ solves the Darcy problem
\begin{equation}
\label{eq:auxiliary}
\left\{
\begin{aligned}
& a(\boldsymbol{\chi}_h, \mathbf{v}_h) + b(\mathbf{v}_h, \rho_h) =  (\widetilde{a}_h(\widetilde{\mathbf{u}}_h, \, \mathbf{v}_h) - a(\widetilde{\mathbf{u}}_h, \, \mathbf{v}_h)) + b(\mathbf{v}_h, z) \qquad & \text{for all $\mathbf{v}_h \in \widetilde{\mathbf{V}_h}$,} \\
&  b(\boldsymbol{\chi}_h, q_h) = 0 \qquad & \text{for all $q_h \in Q_h$.}
\end{aligned}
\right.
\end{equation}
To prove the estimate for the pressure we employ the usual duality argument. Let therefore $\phi$ be the solution of the auxiliary problem 
\[
\left \{
\begin{aligned}
& \Delta \phi = \rho_h  \qquad & \text{on $\Omega$}\\
& \phi = 0               \qquad & \text{on $\partial \Omega$}
\end{aligned}
\right .
\]
that, due to the convexity assumption, satisfies 
\begin{equation}
\label{eq:duality}
\|\phi\|_2 \leq C \, \|\rho_h\|_0
\end{equation}
where the constant $C$ depends only on $\Omega$.
For all $\mathbf{v} \in \widetilde{\mathbf{V}}_h$  let us denote with $\mathbf{v}_I$ its interpolant defined in \eqref{eq:divinterpolantbis} and \eqref{eq:stabilitybis}. Therefore Green formula together with \eqref{eq:auxiliary} yields
\begin{equation}
\label{eq:mupressure}
\begin{split}
\|\rho_h\|^2_0  & = (\rho_h, \, \Delta \phi) = b(\nabla \phi, \, \rho_h) = b((\nabla \phi)_I, \, \rho_h)  \\
               & = (\widetilde{a}_h(\widetilde{\mathbf{u}}_h, \, (\nabla \phi)_I) - a(\widetilde{\mathbf{u}}_h, \, (\nabla \phi)_I)) + b((\nabla \phi)_I, \, z) -   a(\boldsymbol{\chi}_h, (\nabla \phi)_I) \\
               & =: \mu_1((\nabla \phi)_I) + \mu_2((\nabla \phi)_I) +  \mu_3((\nabla \phi)_I). 
\end{split}
\end{equation}
We analyse separately the three terms. For the first one, using the consistency property of $\widetilde{a}_h(\cdot, \cdot)$, the polynomial approximation of $\mathbf{u}$ and $\nabla \phi$, the estimate on the velocity error and \eqref{eq:duality} we get
\begin{equation}
\label{eq:mu1pressure}
\begin{split}
\mu_1((\nabla \phi)_I) &= \widetilde{a}_h(\widetilde{\mathbf{u}}_h, \, (\nabla \phi)_I) - a(\widetilde{\mathbf{u}}_h, \, (\nabla \phi)_I) \\
& = \sum_{K \in \mathcal{T}_h} \left ( \widetilde{a}_h^K(\widetilde{\mathbf{u}}_h, \, (\nabla \phi)_I) - a^K(\widetilde{\mathbf{u}}_h, \, (\nabla \phi)_I) \right) \\
& =  \sum_{K \in \mathcal{T}_h} \left(  \widetilde{a}^K(\widetilde{\mathbf{u}}_h - \mathbf{u}_{\pi}, \, (\nabla \phi)_I -  (\nabla \phi)_{\pi}) - a^K(\widetilde{\mathbf{u}}_h - \mathbf{u}_{\pi}, \, (\nabla \phi)_I  -  (\nabla \phi)_{\pi}) \right) \\
& \leq C \,  \sum_{K \in \mathcal{T}_h}   \|\widetilde{\mathbf{u}}_h - \mathbf{u}_{\pi}\|_{0,K} \|(\nabla \phi)_I -  (\nabla \phi)_{\pi}\|_{0,K}   \\
& \leq C \,  \sum_{K \in \mathcal{T}_h}   (\|\mathbf{u} - \widetilde{\mathbf{u}}_h\|_{0,K}  + \| \mathbf{u} - \mathbf{u}_{\pi}\|_{0,K}) (\|(\nabla \phi) -  (\nabla \phi)_I\|_{0,K} + \|(\nabla \phi) -  (\nabla \phi)_{\pi}\|_{0,K}) \\
& \leq C \,  h^{k-1} (|p|_{k} + h^2|u|_{k+1}) \, h \| \nabla \phi\|_1 \leq C \,  h^{k-1} |p|_{k} \, h \|\phi\|_2 \leq C \, h^{k} (|p|_{k} + h^2|u|_{k+1}) \,\|\rho_h\|_0
\end{split}
\end{equation}
For what concerns the second term we have
\begin{equation}
\label{eq:mu2pressure}
\begin{split}
\mu_2((\nabla \phi)_I) &= b((\nabla \phi)_I, \, z) = b((\nabla \phi)_I - \nabla \phi, \, z)  + b(\nabla \phi, \, z) \\
& \leq C (|\nabla \phi - (\nabla \phi)_I |_1 + \|\phi\|_2) \|z\|_0 \leq C (|\nabla \phi|_1 + \|\phi\|_2) \|z\|_0 \\
& \leq C \, h^{k} |p|_k \|\phi\|_2 \leq C \, h^{k} |p|_k \|\rho_h\|_0. 
\end{split}
\end{equation}
Finally, for the third term we begin by observing that from \eqref{eq:auxiliary} 
\begin{equation}
\label{eq:3atilde}
b(\boldsymbol{\chi}_h, \, \phi) = b(\boldsymbol{\chi}_h, \, \phi - \phi_{\pi}), 
\end{equation}
for all $\phi_{\pi} \in Q_h$, and by the Green formula
\begin{equation}
\label{eq:3btilde}
b(\boldsymbol{\chi}_h, \, \phi) = - a(\boldsymbol{\chi}_h, \, \nabla \phi) = - a(\boldsymbol{\chi}_h, \, \nabla \phi - (\nabla \phi)_I)  - a(\boldsymbol{\chi}_h, \, (\nabla \phi)_I). 
\end{equation}
Therefore, by collecting \eqref{eq:3atilde}, \eqref{eq:3btilde}, and using the previous error estimate,  it holds that
\begin{equation}
\label{eq:mu3pressure}
\begin{split}
\mu_3((\nabla \phi)_I) &=   - a(\boldsymbol{\chi}_h, \, (\nabla \phi)_I) = a(\boldsymbol{\chi}_h, \, \nabla \phi - (\nabla \phi)_I)  + b(\boldsymbol{\chi}_h, \, \phi - \phi_{\pi}) \\
& \leq C (\|\boldsymbol{\chi}_h\|_0 \|\nabla \phi - (\nabla \phi)_I\|_0 + \|\boldsymbol{\chi}_h\|_{\mathbf{V}} \|\phi - \phi_{\pi}\|_0) \\
& \leq C \, h^{k-1} (|p|_{k} + h^2|u|_{k+1}) \, h \, \|\phi\|_2 + h^{k-2} (|p|_{k} + h^2|u|_{k+1}) \, h^2 \, \|\phi\|_2 ) \leq C h^{k} (|p|_{k} + h^2|u|_{k+1}) \|\rho_h\|_0.
\end{split}
\end{equation}
Finally by collecting  \eqref{eq:mu1pressure}, \eqref{eq:mu2pressure} and \eqref{eq:mu3pressure} in \eqref{eq:mupressure} we get the thesis.
\end{proof}

\addcontentsline{toc}{section}{\refname}
\bibliographystyle{plain}
\bibliography{biblio}

\end{document}